\documentclass[a4paper, 12pt]{amsart}
\usepackage[cp1251]{inputenc}
\usepackage[english,russian]{babel}
\usepackage{amsmath,amssymb,a4wide,latexsym,amscd,amsthm}
\usepackage{graphics}

\usepackage[includefoot,top=20mm,left=20mm,right=20mm,bottom=20mm]{geometry}

\theoremstyle{plain}
\newtheorem{theorem}{Теорема}[section]
\newtheorem{lemma}{Лемма}[section]
\newtheorem{proposition}{Утверждение}[section]
\newtheorem{corollary}{Следствие}[section]

\theoremstyle{definition}
\newtheorem{example}{Пример}[section]
\newtheorem{definition}{Определение}[section]

\numberwithin{equation}{section}

\begin{document}

\title[]{О системах подпространств гильбертова пространства,
  удовлетворяющих условиям на угол или коммутации для каждой пары
  подпространств}

\author{А.В.Стрелец, И.С.Фещенко}

\begin{abstract}
Мы изучаем системы подпространств $H_{1},\ldots,H_{n}$ комплексного
гильбертова пространства $H$, удовлетворяющие следующим условиям: для
каждого индекса $i>1$ фиксирован угол $\theta_{1,i}\in(0,\pi/2)$ между
подпространствами $H_{1}$ и $H_{i}$; проекторы на подпространства
$H_{2k}$ и $H_{2k+1}$, $1\leqslant k\leqslant m$, $0\leqslant 2m< n$,
коммутируют; а все остальные пары подпространств $H_{i}$ и $H_{j}$
ортогональны.

Основным инструментом изучения таких систем подпространств является
$G\text{-}$\hspace{0pt}конструкция "--- конструкция системы
подпространств гильбертова пространства по ее оператору Грама.
\end{abstract}

\maketitle

\section{Введение}

\subsection{Системы подпространств}

Изучение систем $L=(V;V_{1},\ldots,V_{n})$ подпространств
$V_{1},\ldots,V_{n}$ конечномерного линейного пространства~$V$, $n\in\mathbb{N}$, в
частности, описание неразложимых четверок подпространств в $V$
\cite{Gelfand}, описание неразложимых представлений в пространстве $V$
конечных частично упорядоченных множеств (см., например,
\cite{Nazarova}), и т.д. являются классическими задачами алгебры
(см. библиографию в \cite{SamStr}).

Пусть $H$ "--- комплексное гильбертово пространство, $H_{k}$,
$1\leqslant k\leqslant n$, "--- набор его подпространств.  Изучение
систем подпространств $S=(H;H_{1},\ldots,H_{n})$ гильбертова
пространства $H$ (или, что тоже самое, наборов соответствующих
ортопроекторов $P_{1},\ldots,P_{n}$) является важной задачей
функционального анализа, которой посвящены многочисленные публикации
(см., например, \cite{SamStr} и библиографию там).

Описание всех неразложимых систем подпространств~$S$ с точностью до
унитарной эквивалентности хорошо известно в случае $n\leqslant2$. Так,
в случае~$n=1$ любая неразложимая система $S$ унитарно эквивалентна
одной из систем $S_0=(\mathbb{C};0)$ и $S_1=(\mathbb{C}; \mathbb{C})$;
а в случае $n=2$ (см., например, \cite{Halmos}) с точностью до
унитарной эквивалентности существуют четыре неразложимые пары
подпространств $S_{00}=(\mathbb{C};0,0)$, $S_{01}=(\mathbb{C};0,
\mathbb{C})$, $S_{10}=(\mathbb{C};\mathbb{C},0)$,
$S_{11}=(\mathbb{C};\mathbb{C},\mathbb{C})$ и семейство неразложимых
систем одномерных подпространств в двумерном пространстве
$S_\varphi=(H;H_1,H_2)$, $\varphi\in(0,\frac\pi2)$, таких, что в
некотором ортонормированном базисе~$\{e_1,e_2\}$ в $H$,
подпространство~$H_1$ порождено вектором $e_1$, а
подпространство~$H_2$ порождено вектором $x=\cos\varphi
e_1+\sin\varphi e_2$.

В случае $n\geqslant3$ задача описания неразложимых систем
подпространств~$S$ с точностью до унитарной эквивалентности является
$*\text{-}$\hspace{0pt}дикой (см.~\cite{KrSam:PAMS_2000, OstSam,
  KrSam:FAP_1980}). $*\text{-}$\hspace{0pt}Дикой является даже задача
об описании троек подпространств $S=(H;H_1,H_2,H_3)$ таких что
$H_2\bot H_3$ (о $*\text{-}$\hspace{0pt}диких задачах
см.~\cite{KrSam:PAMS_2000, OstSam}).

Таким образом, естественно
выделить тот или иной класс систем подпространств и, по возможности,
описать с точностью до унитарной эквивалентности все неразложимые
системы из выбранного класса.

\subsection{Некоторые классы систем подпространств}

В работе \cite{TempLieb} физики H. N. V. Temperley и E. H. Lieb
ввели алгебры
\begin{align*}
  \mathbb{C}\langle
  p_{1},p_{2},\ldots,p_{n}\:|\:
  &p_{j}^{2}=p_{j},\, j=1,2,\ldots,n;\\
  &p_{i}p_{j}p_{i}=\nu p_{i},\, |i-j|=1;\\
  &p_{i}p_{j}=p_{j}p_{i},\, |i-j|\geqslant 2
  \rangle,
\end{align*}
$\nu\in\mathbb{C}$, в связи с изучением моделей статистической
физики. В случае $\nu=\tau_0^2\in(0,1)$ такие алгебры можно
рассматривать как $*\text{-}$\hspace{0pt}алгебры, если определить в
них инволюцию равенствами $p_j^*=p_j$, $1\leqslant j\leqslant n$.  Пусть $\pi$ "--- некоторое
$*\text{-}$\hspace{0pt}представление такой
$*\text{-}$\hspace{0pt}алгебры в гильбертовом пространстве $H$, а
$H_i$ "--- образы ортопроекторов $P_i=\pi(p_i)$. Таким образом мы
получили систему подпространств $S=(H;H_1,\ldots,H_n)$, которая
удовлетворяет следующим условиям:
\begin{enumerate}
\item <<соседние>> пары подпространств \emph{расположены друг
  относительно друга под углом~$\theta_0$}, $\tau_0=\cos \theta_0$,
  т.е.  $P_iP_{i+1}P_i=\tau^2_0P_i$,
  $P_{i+1}P_iP_{i+1}=\tau^2_0P_{i+1}$, $i=1,\ldots, n-1$;
\item остальные пары подпространств \emph{<<коммутируют>>}, т.е.
  $P_iP_j=P_jP_i$.
\end{enumerate}
Рассмотрим граф $A_{n}$ с множеством вершин $\{1,2,\ldots,n\}$ и
множеством ребер $\{i,i+1\}$, $1\leqslant i\leqslant n-1$. Тогда
условия на угол между подпространствами соответствуют парам вершин
$i,j$, соединенных ребром в $A_{n}$, а условия коммутации "--- парам
вершин $i,j$, не соединенных ребром.

Используя это наблюдение, можно определить класс систем
подпространств, связанный с графом $\Gamma$ и функцией $\tau$ на его
ребрах. Пусть $\Gamma$ "--- граф без петель и кратных ребер с
множеством вершин $\{1,2,\ldots,n\}$. Обозначим через $E$ множество
ребер $\Gamma$, а через $\overline{E}$ "--- множество пар вершин
$\{i,j\}$, не соединенных ребром в $\Gamma$. Пусть на ребрах $\Gamma$,
заданы функции
\begin{equation*}
  \theta:E\to(0,\pi/2):\{i,j\}\mapsto \theta_{\{i,j\}}
  \quad\text{и}\quad
  \tau=\cos\theta:E\to(0,1):\{i,j\}\mapsto \tau_{\{i,j\}}.
\end{equation*}
Через $Sys(\Gamma,\tau)$ обозначим множество систем подпространств
$S=(H;H_1,\ldots,H_n)$ таких, что
\begin{enumerate}
\item
  если $\{i,j\}\in E$, то подпространства $H_{i},H_{j}$ расположены
  друг относительно друга под углом $\theta_{\{i,j\}}$, т.е.
  $P_iP_jP_i=\tau_{\{i,j\}}^{2}P_i$ и
  $P_jP_iP_j=\tau_{\{i,j\}}^{2}P_j$;
\item
  если $\{i,j\}\in\overline{E}$, то подпространства $H_i,H_j$
  <<коммутируют>>, т.е.  $P_iP_j=P_jP_i$.
\end{enumerate}
Системы $S\in Sys(\Gamma,\tau)$ можно рассматривать как
$*\text{-}$\hspace{0pt}представления соответствующей
$*\text{-}$\hspace{0pt}алгебры
\begin{align*}
  \mathcal{TL}_{\Gamma,\tau}=\mathbb{C}\langle
  p_{1},p_{2},\ldots,p_{n}\:|\:
  &p_{j}^{2}=p_{j}^{*}=p_{j},\, j=1,2,\ldots,n;\\
  &p_{i}p_{j}p_{i}=\tau_{\{i,j\}}^{2} p_{i},\, \{i,j\}\in E;\\
  &p_{i}p_{j}=p_{j}p_{i},\, \{i,j\}\in\overline{E}
  \rangle.
\end{align*}
Отметим, что если <<забыть>> об инволюции, т.е. в определении
$\mathcal{TL}_{\Gamma,\tau}$ условия $p_{j}^{2}=p_{j}^{*}=p_{j}$
заменить на $p_{j}^{2}=p_{j}$, то определенная таким образом алгебра
будет \emph{проективной алгеброй} (см. \cite{Graham}, раздел 6).

Более узкий класс систем подпространств получится, если для каждой
пары вершин $i,j$, не соединенных ребром в $\Gamma$, условие
коммутации усилить, заменив на условие \emph{ортогональности}
$P_iP_j=P_jP_i=0$. Множество таких систем подпространств "---
<<простых>> систем подпространств "--- обозначим через
$Sys(\Gamma,\tau,\bot)$. Такие системы можно рассматривать как
$*\text{-}$\hspace{0pt}представления соответствующей
$*\text{-}$\hspace{0pt}алгебры $\mathcal{TL}_{\Gamma,\tau,\bot}$
(которая является фактор-алгеброй $*\text{-}$\hspace{0pt}алгебры
$\mathcal{TL}_{\Gamma,\tau}$). Изучению класса систем подпространств
$Sys(\Gamma,\tau,\bot)$ и $*\text{-}$\hspace{0pt}алгебр
$\mathcal{TL}_{\Gamma,\tau,\bot}$ посвящена серия работ
(см. обзор~\cite{SamStr}).

Естественным образом возникает класс систем подпространств, занимающий
<<промежуточное>> положение между $Sys(\Gamma,\tau)$ и
$Sys(\Gamma,\tau,\bot)$. Пусть $E^{c}$ "--- некоторое подмножество
множества $\overline{E}$.  Обозначим через $Sys(\Gamma,E^{c},\tau)$
множество систем подпространств $S=(H;H_1,\ldots,H_n)$ таких, что
\begin{enumerate}
\item если $\{i,j\}\in E$, то подпространства $H_{i},H_{j}$
  расположены друг относительно друга под углом $\theta_{\{i,j\}}$;
\item если $\{i,j\}\in E^{c}$, то подпространства $H_i,H_j$
  <<коммутируют>>;
\item если $\{i,j\}\in \overline{E}\setminus E^{c}$, то подпространства $H_i,H_j$
  ортогональны.
\end{enumerate}

В настоящей работе мы изучаем классы систем
подпространств~$Sys(K_{1,N},E_m^c,\tau)$, где $N\in\mathbb{N}$,
$m\in\mathbb{N}\cup \{0\}$, $2m\leqslant N$, $K_{1,N}$ "--- <<звезда>>
с $N$ лучами, $\tau$ "--- произвольная функция на ребрах $K_{1,N}$, а множество
$E_{m}^c$ состоит из пар вершин $\{2k, 2k+1\}$, $1\leqslant k\leqslant
m$.  Перед тем как более аккуратно сформулировать задачу и основные
результаты, напомним некоторые необходимые определения.

\subsection{Основные определения}
Систему подпространств $S=(H;H_{1},\ldots,H_{n})$ называют
\emph{разложимой}, если существует ортогональное разложение
$H=H'\oplus H''$ в прямую сумму ненулевых подпространств $H',H''$ и
системы подпространств $S'=(H';H_{1}',\ldots,H_{n}')$,
$S''=(H'';H_{1}'',\ldots,H_{n}'')$ такие, что $H_{k}=H_{k}'\oplus
H_{k}''$ для всех $1\leqslant k\leqslant n$. Система подпространств
$S$ называется \emph{неразложимой} если она не является разложимой.
Хорошо известно, что неразложимость системы $S$ равносильна ее
\emph{неприводимости}, т.е. выполнению следующего условия: если
ограниченный линейный оператор $A:H\to H$ удовлетворяет
$AP_{k}=P_{k}A$ при $1\leqslant k\leqslant n$, то $A=\lambda I$ для
некоторого $\lambda\in\mathbb{C}$.

Две системы подпространств $S=(H;H_{1},\ldots,H_{n})$ и
$S'=(H';H_{1}',\ldots,H_{n}')$ подпространств в $H$ и $H'$ называются
\emph{унитарно эквивалентными}, если существует унитарный оператор
$U:H\to H'$ такой, что $H_{k}'=U(H_{k})$ для всех $1\leqslant
k\leqslant n$. Последнее условие равносильно $P_{k}'=UP_{k}U^{*}$,
т.е.  $UP_{k}=P_{k}'U$.

Для системы подпространств $S$, вектор $(\dim H;\dim H_{1},\ldots,\dim
H_{n})$, компонентами которого являются мощности, называют
\emph{обобщенной размерностью} системы подпространств~$S$.  В
дальнейшем, чтобы не усложнять обозначения, если для системы $S$ все
$\dim H_{k}$ равны, то обобщенной размерностью $S$ будем называть
вектор $(\dim H;\dim H_{1})$.

Система подпространств $S$ называется \emph{нулевой}, если $H_{i}=0$
для всех $1\leqslant i\leqslant n$. В противном случае система
называется ненулевой.

\subsection{Постановка задачи и основные результаты}
Пусть $m$, $r$ "--- неотрицательные целые числа, положим
$N=2m+r$. Рассмотрим <<звезду>> с $N$ лучами, т.е. граф~$K_{1,N}$,
множество вершин которого $V = \{1, 2, \ldots, N+1 \}$ занумеровано
таким образом, что вершина $1$ соединена со всеми остальными
вершинами. Таким образом, множество ребер $E$ равно
$\{\{1,k\}\,|\,k\in V\setminus\{1\} \}$. Через $E^{c}_m$ обозначим
множество пар вершин $\{\{2k,2k+1\}\,|\, 1\leqslant k\leqslant m\}$
(на рисунке эти пары вершин соединены пунктиром):
\vspace{10pt}

\begin{center}
  \begin{picture}(200,135)
    \put(20,50){\circle*{3}}
    \put(20,90){\circle*{3}}
    \put(80,70){\circle*{3}}
    \put(140,90){\circle*{3}}
    \put(140,50){\circle*{3}}
    \put(60,130){\circle*{3}}
    \put(100,130){\circle*{3}}

    \put(90,6){\circle*{3}}
    \put(124,18){\circle*{3}}
    \put(44,16){\circle*{3}}

    \qbezier(20,50)(50,60)(80,70)
    \qbezier(20,90)(50,80)(80,70)

    \qbezier(80,70)(70,100)(60,130)
    \qbezier(80,70)(90,100)(100,130)

    \qbezier(80,70)(110,80)(140,90)
    \qbezier(80,70)(110,60)(140,50)

    \qbezier(80,70)(102,44)(124,18)

    \qbezier(80,70)(85,38)(90,6)

    \qbezier(80,70)(62,43)(44,16)

    \put(68,78){1}
    \put(10,48){2}
    \put(10,88){3}
    \put(50,130){4}
    \put(105,130){5}
    \put(145,88){$2m$}
    \put(145,48){$2m+1$}

    \put(128,14){$2m+2$}
    \put(83,-6){$2m+3$}
    \put(20,6){$2m+1+r$}

    \put(100,90){$\ddots$}

    \put(65,25){$\ddots$}

    \multiput(20,50)(0,4){10}{\circle*{1}}
    \multiput(60,130)(4,0){10}{\circle*{1}}
    \multiput(140,50)(0,4){10}{\circle*{1}}

  \end{picture}
\end{center}

Пусть каждому ребру $\{1,k\}$ сопоставлен угол
$\theta_{\{1,k\}}\in(0,\pi/2)$, т.е. задана функция
$\theta:E\to(0,\pi/2)$. Определим функцию $\tau=\cos\theta:E\to(0,1)$,
т.е.  $\tau_{\{1,k\}}=\cos\theta_{\{1,k\}}\in(0,1)$.

Введенные граф~$K_{1,N}$, множество~$E_m^c$ и функция $\tau$ позволяют
<<наглядно>> сформулировать условия, которым удовлетворяют системы
подпространств~$S=(H;H_{1},\ldots,H_{n})$, $n=N+1$, рассматриваемые в
этой статье.

\textbf{(Ang): Условия на углы} (соответствующие парам вершин,
соединенных ребром).  Для каждого $k=2,3,\ldots,N+1$ подпространства
$H_{1}$ и $H_{k}$ расположены друг относительно друга под углом
$\theta_{\{1,k\}}$, т.е.  $P_{1}P_{k}P_{1}=\tau_{\{1,k\}}^{2}P_{1}$ и
$P_{k}P_{1}P_{k}=\tau_{\{1,k\}}^{2}P_{k}$.

\textbf{(Com): Условия коммутации} (соответствующие парам вершин,
соединенных пунктиром).  Для каждого $k=1,2,\ldots,m$ ортопроекторы
$P_{2k}$ и $P_{2k+1}$ коммутируют, т.е.
$P_{2k}P_{2k+1}=P_{2k+1}P_{2k}$.

\textbf{(Ort): Условия ортогональности} (соответствующие парам вершин,
не соединенных ребром или пунктиром).  Если пара различных вершины
$i,j$ не соединена ребром, и эта пара не принадлежит множеству
$E^c_m$, то соответствующие подпространства $H_{i}$ и $H_{j}$
ортогональны, т.е.  $P_{i}P_{j}=0$.

Отметим, что в случае $m=0$ изучаемая система подпространств является
<<простой>> системой подпространств, связанной с графом $K_{1,N}$ и функцией $\tau$
на его ребрах.

Системы подпространств, удовлетворяющие приведенным выше условиям,
можно рассматривать как $\ast\text{-}$\hspace{0pt}представления в
гильбертовом пространстве
$\ast\text{-}$\hspace{0pt}\nolinebreakалгебры
$\mathcal{TL}(K_{1,N},E^{c}_m,\tau)$, определенной образующими и
соотношениями:
\begin{align*}
  \mathcal{TL}(K_{1,N},E^c_m,\tau)=
  \mathbb{C}\langle
  p_{1},p_{2},\ldots,p_{N+1}\:|\:
  &p_{j}^{2}=p_{j}^{*}=p_{j},\; j\in V\\
  &p_{i}p_{j}p_{i}=\tau_{\{i,j\}}^{2}p_{i},\; \{i,j\}\in E\\
  &p_{i}p_{j}=p_{j}p_{i},\; \{i,j\}\in E^c_m,\\
  &p_{i}p_{j}=0,\; \{i,j\}\not\in E\cup E^c_m
  \rangle.
\end{align*}
Взаимно однозначное соответствие между системами
подпространств $S$, удовлетворяющими условиям (Ang), (Com), (Ort), и
$\ast\text{-}$\hspace{0pt}представлениями $\pi$
$\ast\text{-}$\hspace{0pt}алгебры $\mathcal{TL}(K_{1,N},E^c_m,\tau)$ в
гильбертовом пространстве $H$ задается равенством
$H_{k}=\mathrm{Im}\,\pi(p_{k})$, $k\in V$.

В обозначении $\ast\text{-}$\hspace{0pt}алгебры
$\mathcal{TL}(K_{1,N},E^c_m,\tau)$ буквы $\mathcal{TL}$ выбраны в
честь физиков H. N. V. Temperley и E. H. Lieb'а.

Основными результатами нашей работы являются:
\begin{enumerate}
\item
  описание (с точностью до унитарной эквивалентности) всех систем
  подпространств $S$, удовлетворяющих (Ang), (Com), (Ort)
  (см. подраздел~\ref{SS:all});
\item
  описание (с точностью до унитарной эквивалентности) всех
  неприводимых систем подпространств $S$, удовлетворяющих (Ang),
  (Com), (Ort) (см. теорему~\ref{T:main} и подраздел~\ref{SS:irred}).
\end{enumerate}

Отметим, что для $m\geqslant3$ при вариациях параметров
$\tau_{\{1,k\}}$ возникают три возможные ситуации: существует
конечное число унитарно неэквивалентных неприводимых систем
подпространств (\emph{конечная} задача); унитарно неэквивалентных
неприводимых систем подпространств бесконечное число, но их все еще
можно описать (\emph{ручная} задача); задача описания всех систем
подпространств с точностью до унитарной эквивалентности является
<<безнадежной>> в определенном смысле (\emph{дикая} задача).

В разделе~\ref{S:construction} приводится
$G\text{-}$\hspace{0pt}конструкция, которая является основным
инструментом, используемым для описания систем подпространств. Сама
$G\text{-}$\hspace{0pt}конструкция и результаты
раздела~\ref{S:construction} за исключением результатов
подраздела~\ref{sec:irreducibility} сформулированы для произвольных
систем подпространств. Утверждения подраздела~\ref{sec:irreducibility}
могут быть усилены, но с одной стороны это привело бы к усложнению
доказательств, а с другой стороны и в таком варианте они могут быть
использованы для изучения класса систем подпространств намного более
широкого, чем тот, который изучается в настоящей работе.

Авторы надеются, что $G\text{-}$\hspace{0pt}конструкция позволит в
дальнейших исследованиях получить описания других классов систем
подпространств. Например, одним из интересных для изучения классов
является класс систем подпространств, который задается графом
$K_{1,N}$, функцией $\tau$ на ребрах и множеством пар вершин
(<<пунктирных>> ребер) $E^c$, более сложным, чем $E^c_m$.

\subsection{Обозначения}  В данной работе мы рассматриваем комплексные
гильбертовы пространства, которые обозначаем буквами $H$, $M$,
$K$. Отметим, что мы не накладываем дополнительных условий на
размерность гильбертова пространства. Чтобы не усложнять обозначения,
скалярное произведение, как правило, обозначается
$\langle\cdot,\cdot\rangle$. В различных гильбертовых пространствах
скалярное произведение может обозначаться одинаково, если это не
приводит к недоразумению.

\section{$G\text{-}$\hspace{0pt}конструкция}\label{S:construction}

\subsection{$G\text{-}$\hspace{0pt}конструкция системы подпространств гильбертова пространства.}

Пусть $H_{0,k}$, $1\leqslant k\leqslant n$, "--- набор ненулевых
гильбертовых пространств. Определим гильбертово пространство
$\widetilde{H}=H_{0,1}\oplus\ldots\oplus H_{0,n}$ и будем обозначать
через $\langle\cdot,\cdot\rangle_{0}$ скалярное произведение в
нем. Пусть $\Gamma_{k}:H_{0,k}\to\widetilde{H}$ "--- естественное
вложение пространства $H_{0,k}$ в пространство~$\widetilde{H}$, то
есть $\Gamma_{k}x=(0,\ldots,0,x,0,\ldots,0)$, где $x$ стоит на
$k\text{-}$\hspace{0pt}том месте. Тогда оператор
$\Gamma_{k}^{*}:\widetilde{H}\to H_{0,k}$ является оператором
выделения $k\text{-}$\hspace{0pt}ой координаты:
$\Gamma_{k}^{*}(x_{1},\ldots,x_{n})=x_{k}$.  Определим
\begin{equation*}
  \widetilde{H}_{k}=\mathrm{Im}\,\Gamma_{k},\quad 1\leqslant k\leqslant n.
\end{equation*}

Пусть $B:\widetilde{H}\to\widetilde{H}$ "--- ограниченный
неотрицательный самосопряженный оператор, причем для его блочного
разложения $B=(B_{i,j}:H_{0,j}\to H_{0,i},\,1\leqslant i,j\leqslant
n)$ выполнено
\begin{equation}\label{eq:ident}
  B_{k,k}=I_{H_{0,k}},\quad 1\leqslant k\leqslant n.
\end{equation}

Определим $\widetilde{H}_{0}=\mathrm{Ker}\,B$.  Используя оператор
$B$, зададим скалярное произведение в линейном пространстве
$\widetilde{H}/\widetilde{H}_{0}$ с помощью равенства:
\begin{equation*}
  \langle x+\widetilde{H}_{0},y+\widetilde{H}_{0}\rangle=
  \langle Bx,y\rangle_{0}, \quad x,y\in\widetilde{H}.
\end{equation*}
Очевидно, что это определение корректно, так как не зависит от выбора
представителей классов эквивалентности.  Пусть $H$ "--- пополнение
пространства $\widetilde{H}/\widetilde{H}_{0}$ относительно введенного
скалярного произведения.

Определим ограниченный линейный оператор $\rho:\widetilde{H}\to H$
равенством
\begin{equation*}
  \rho(x)=x+\widetilde{H}_{0}.
\end{equation*}
Ясно, что $\mathrm{Im}\, \rho=\widetilde{H}/\widetilde{H}_{0}$.
Положим $H_{k}=\rho(\widetilde{H}_{k})=\{z+\widetilde{H}_{0}\,\mid\,
z\in\widetilde{H}_{k}\}$, $1\leqslant k\leqslant n$. Так как для
произвольного $z\in\widetilde{H}_{k}$
\begin{equation*}
  \|z+\widetilde{H}_{0}\|=\sqrt{\langle Bz,z\rangle_{0}}=\|z\|_{0},
\end{equation*}
то $H_{k}$ является подпространством пространства~$H$.  Кроме того,
\begin{equation*}
  \rho_{k}=\rho \upharpoonright_{\widetilde{H}_{k}}:\widetilde{H}_{k}\to H_{k},\quad 1\leqslant k\leqslant n,
\end{equation*}
является унитарным оператором. Далее, поскольку
$H_{1}+\ldots+H_{n}=\rho(\widetilde{H})=\widetilde{H}/\widetilde{H}_{0}$,
то $H_{1}+\ldots+H_{n}$ плотно в $H$.

Систему подпространств $(H;H_{1},\ldots,H_{n})$, полученную в
результате применения приведенной выше конструкции, будем
обозначать~$\mathcal{G}(H_{0,1},\ldots,H_{0,n};B)$, а саму конструкцию будем называть
$G\text{-}$\hspace{0pt}конструкцией.

\subsection{Произвольная система подпространств как результат $G\text{-}$\hspace{0pt}конструкции}

\begin{definition}
  Пусть $K$ "--- гильбертово пространство, $S=(K;K_{1},\ldots,K_{n})$
  "--- система его подпространств.  Обозначим $Q_{i}$ ортопроектор на
  $K_{i}$, $1\leqslant i\leqslant n$.  Оператор
  $G(S):\oplus_{i=1}^{n}K_{i}\to\oplus_{i=1}^{n}K_{i}$, заданный своим
  блочным разложением $G_{i,j}=Q_{i}\upharpoonright_{K_{j}}:K_{j}\to
  K_{i}$, $1\leqslant i,j\leqslant n$, называют оператором Грама
  системы подпространств $S$.
\end{definition}

\begin{proposition}\label{P:canonicalconstr}
 Пусть $S=(K;K_{1},\ldots,K_{n})$ "--- система ненулевых подпространств
 гильбертова пространства
 $K$, причем
 $K_{1}+\ldots+K_{n}$ плотно в $K$. Тогда система подпространств
 $\mathcal{G}(K_{1},\ldots,K_{n};G(S))$ унитарно эквивалентна системе
 $S$.
\end{proposition}

\begin{proof}
  Из определения~$G=G(S)$ видно, что $G_{i,i}=I_{K_{i}}$ и
  $G_{i,j}^{*}=G_{j,i}$, $1\leqslant i,j\leqslant n$. Покажем, что
  оператор~$G$ неотрицателен. Скалярное
  произведение в $K$ обозначим $\langle\cdot,\cdot\rangle$. Для
  произвольных $x_{j}\in K_{j}$, $y_{i}\in K_{i}$ имеем
  \begin{equation*}
    \langle G_{i,j}x_{j},y_{i}\rangle=\langle Q_{i}x_{j},y_{i}\rangle=\langle x_{j},y_{i}\rangle.
  \end{equation*}
  Тогда для произвольных $x=(x_{1},\ldots,x_{n})\in\widetilde{H}$ и
  $y=(y_{1},\ldots,y_{n})\in\widetilde{H}$ верно равенство
  \begin{equation*}
    \langle Gx,y\rangle_{0}=\langle\sum_{k=1}^{n}x_{k},\sum_{k=1}^{n}y_{k}\rangle.
  \end{equation*}
  Отсюда следует, что~$G$ неотрицателен, а его ядро $\mathrm{Ker}\,
  G$ состоит из векторов~$x\in \widetilde{H}$ таких, что
  $x_1+\dots+x_n=0$.

  Определим оператор $U:\rho(\widetilde{H})\to\sum_{i=1}^{n}K_{i}$
  равенством
  \begin{equation*}
    U(x+\widetilde{H}_{0})=\sum_{k=1}^{n}x_{k},
    \quad x=(x_1,\dots,x_n)\in\widetilde{H},
    \quad x_{i}\in K_{i},\quad 1\leqslant i\leqslant n.
  \end{equation*}
  Это определение корректно, поскольку $\widetilde{H}_0=\mathrm{Ker}\,G$.

  Так как
  \begin{equation*}
    \langle x+\widetilde{H}_0, y+\widetilde{H}_0\rangle =
    \langle Gx, y\rangle_0 = \langle\sum_{i=1}^{n}x_{i},\sum_{i=1}^{n}y_{i}\rangle,
  \end{equation*}
  и сумма $K_1+\dots+K_n$ плотна в $K$, то $U$ "--- линейный оператор,
  сохраняющий скалярное произведение, образ которого плотен в $K$.
  Следовательно, оператор~$U$ единственным образом продолжается по
  непрерывности до унитарного оператора $\overline{U}:H\to K$.
  Поскольку $\overline{U}(H_{i})=K_{i}$ для всех $1\leqslant
  i\leqslant n$, то утверждение доказано.
\end{proof}

\subsection{Критерий унитарной эквивалентности систем подпространств полученных в результате $G\text{-}$\hspace{0pt}конструкции}

\begin{proposition}\label{P:unitequiv}
  Системы подпространств
  \begin{equation*}
    \mathcal{G}(H_{0,1},\ldots,H_{0,n};B)\quad \text{ и }\quad
    \mathcal{G}(H'_{0,1},\ldots,H'_{0,n};B')
  \end{equation*}
  унитарно эквивалентны тогда и только тогда, когда существует набор
  унитарных операторов $U_{0,k}:H'_{0,k}\to H_{0,k}$, $1\leqslant
  k\leqslant n$, такой, что для произвольных $i,j$ выполнено равенство
  \begin{equation}\label{eq:unitary}
    B'_{i,j}=U_{0,i}^{*}B_{i,j}U_{0,j}.
  \end{equation}
  Если ввести унитарный оператор
  \begin{equation*}
    \widetilde{U}=\mathrm{diag}\,(U_{0,1},\ldots,U_{0,n}):\widetilde{H}'\to \widetilde{H},
  \end{equation*}
  то систему равенств~\eqref{eq:unitary} можно записать в виде
  $B'=\widetilde{U}^{*}B\widetilde{U}$.
\end{proposition}

\begin{proof}
  \textbf{1.}
  Пусть системы подпространств
  \begin{equation*}
    (H,H_{1},\ldots,H_{n})=\mathcal{G}(H_{0,1},\ldots,H_{0,n};B)
    \quad\text{ и }\quad
    (H',H'_{1},\ldots,H'_{n})=\mathcal{G}(H'_{0,1},\ldots,H'_{0,n};B')
  \end{equation*}
  унитарно эквивалентны. Тогда существует унитарный оператор $U:H'\to
  H$, такой, что $U(H_{k}^{\prime})=H_{k},\,1\leq k\leq n$.
  Определим
  унитарный оператор $U_k=U\upharpoonright_{H'_{k}}:H'_{k}\to
  H_{k}$. Определим унитарный оператор $U_{0,k}:H'_{0,k}\to H_{0,k}$
  равенством
  \begin{equation*}
    U_{0,k}=\Gamma_{k}^{*}\rho_{k}^{-1}U_{k}\rho_{k}^{\prime}\Gamma_{k}^{\prime}
  \end{equation*}
  Тогда для каждого $1\leqslant k\leqslant n$ диаграмма
  \begin{equation}\label{eq:diag}
    \begin{CD}
      H'_{0,k} @>U_{0,k}>> H_{0,k}\\
      @V\Gamma'_{k}VV @V\Gamma_{k}VV \\
      \widetilde{H}_{k}^{\prime} @>>> \widetilde{H}_{k}\\
      @V\rho'_{k}VV @V\rho_{k}VV\\
      H'_{k}@>U_k>> H_{k}
    \end{CD}
  \end{equation}
  коммутативна.

  Покажем, что выполнено равенство
  $B_{i,j}^{\prime}=U_{0,i}^{*}B_{i,j}U_{0,j}$, $1\leqslant
  i,j\leqslant n$. В силу коммутативности диаграммы~\eqref{eq:diag},
  для любого $x\in H'_{0,k}$ верно равенство
  \begin{equation*}
    \Gamma_k U_{0,k}x+\widetilde{H}_0=
    U_{k}(\Gamma'_kx+\widetilde{H}'_0)=U(\Gamma'_{k}x+\widetilde{H}'_0).
  \end{equation*}
  Пусть $x\in H'_{0,j}$, $y\in H'_{0,i}$, тогда верны равенства
  \begin{align*}
    \langle B'_{i,j}x,y\rangle&
    =\langle B'\Gamma'_j x,\Gamma'_i y\rangle_0
    =\langle \Gamma'_j x+\widetilde{H}'_0,\Gamma'_i y+\widetilde{H}'_0\rangle\\
    \langle B_{i,j}U_{0,j}x,U_{0,i}y\rangle&
    =\langle B\Gamma_j U_{0,j}x,\Gamma_i U_{0,i}y\rangle_0
    =\langle \Gamma_j U_{0,j}x+\widetilde{H}_0,\Gamma_i U_{0,i}y+\widetilde{H}_0\rangle,
  \end{align*}
  следовательно $\langle B'_{i,j}x,y\rangle=\langle
  U_{0,i}^*B_{i,j}U_{0,j}x,y\rangle$, то есть $B'_{i,j}=U_{0,i}^*B_{i,j}U_{0,j}$.

  \textbf{2.} Наоборот, пусть существует набор унитарных операторов
  $U_{0,k}:H'_{0,k}\to H_{0,k}$, $1\leqslant k\leqslant n$, такой, что
  $B'=\widetilde{U}^{*}B\widetilde{U}$. Определим оператор
  $U:\rho'(\widetilde{H}')\to \rho(\widetilde{H})$ равенством
  \begin{equation*}
   U(x+\widetilde{H}'_{0})=\widetilde{U}x+\widetilde{H}_{0},\quad x\in\widetilde{H}'.
  \end{equation*}
  Это определение корректно, так как
  $\widetilde{U}(\mathrm{Ker}\,B')=\mathrm{Ker}\,B$.

  Для произвольных
  $x,y\in\widetilde{H}'$ имеем:
  \begin{equation*}
  \langle x+\widetilde{H}'_0,y+\widetilde{H}'_0\rangle=
  \langle B'x,y\rangle_{0}=
  \langle \widetilde{U}^{*}B\widetilde{U}x,y\rangle_{0}=
  \langle B\widetilde{U}x,\widetilde{U}y\rangle_{0}=
  \langle \widetilde{U}x+\widetilde{H}_{0},\widetilde{U}y+\widetilde{H}_{0}\rangle.
  \end{equation*}
  Таким образом, $U$ "--- линейный оператор, сохраняющий скалярное
  произведение, образ которого
  $U(\rho'(\widetilde{H}'))=\rho(\widetilde{H})$ плотен в $H$.
  Поэтому существует единственное продолжение по непрерывности $U$ до
  унитарного оператора $\overline{U}:H'\to H$.  Ясно, что для всех
  $1\leqslant k\leqslant n$, $\overline{U}(H'_{k})=H_{k}$, таким
  образом утверждение доказано.
\end{proof}

\subsection{Связь свойств системы подпространств $\mathcal{G}(H_{0,1},\ldots,H_{0,n};B)$
со свойствами оператора $B$}\label{SS:properties}

Пусть система подпространств
$S=(H;H_{1},\ldots,H_{n})=\mathcal{G}(H_{0,1},\ldots,H_{0,n};B)$.
Обозначим $P_{i}$ ортопроектор на $H_{i}$, $1\leqslant i\leqslant n$.
Пусть $G=G(S)$ "--- оператор Грама системы $S$,
$G_{i,j}=P_{i}\upharpoonright_{H_{j}}:H_{j}\to H_{i}$, $1\leqslant
i,j\leqslant n$.

\begin{proposition}\label{P:connection}
  Существует набор унитарных операторов $U_{0,k}:H_{k}\to H_{0,k}$,
  $1\leqslant k\leqslant n$, такой, что
  $G_{i,j}=U_{0,i}^{*}B_{i,j}U_{0,j}$ $1\leqslant i,j\leqslant n$.
\end{proposition}
\begin{proof}
  Из утверждения \ref{P:canonicalconstr} следует, что $S$ унитарно
  эквивалентна системе подпространств
  $\mathcal{G}(H_{1},\ldots,H_{n};G)$.  Теперь из утверждения
  \ref{P:unitequiv} получаем нужное.
\end{proof}

Утверждение~\ref{P:connection} позволяет связать свойства системы
подпространств $S$ со свойствами оператора~$B$.  Во всех следующих
примерах через $\alpha,\beta$ мы обозначаем пару различных индексов из
множества~$\{1,2,\ldots,n\}$.

\begin{example} \emph{Условие ортогональности.}  Подпространства
  $H_{\alpha}$ и $H_{\beta}$ ортогональны тогда и только тогда, когда
  $P_{\alpha}P_{\beta}=0$, что равносильно
  условию~$G_{\alpha,\beta}=0,$ что, в свою очередь,
  равносильно~$B_{\alpha,\beta}=0.$
\end{example}

\begin{example} \emph{Условие на угол между подпространствами.}
Пусть $\theta_{0}\in[0,\pi/2)$. Определим
  $\tau_{0}=\cos\theta_{0}$. Подпространства $H_{\alpha},H_{\beta}$
  расположены друг относительно друга под углом $\theta_{0}$ тогда и
  только тогда, когда
\begin{equation*}
  P_{\alpha}P_{\beta}P_{\alpha}=\tau_{0}^{2}P_{\alpha},\,
  P_{\beta}P_{\alpha}P_{\beta}=\tau_{0}^{2}P_{\beta},
\end{equation*}
т.е.
\begin{equation*}
  G_{\alpha,\beta}G_{\beta,\alpha}=\tau_{0}^{2}I_{H_{\alpha}},\,
  G_{\beta,\alpha}G_{\alpha,\beta}=\tau_{0}^{2}I_{H_{\beta}}.
\end{equation*}
Поскольку $G_{i,j}=U_{0,i}^{*}B_{i,j}U_{0,j},\,1\leqslant i,j\leqslant
n$ для некоторого набора унитарных операторов $U_{0,k}:H_{k}\to
H_{0,k},\,1\leqslant k\leqslant n$, то последнее условие равносильно
\begin{equation*}
  B_{\alpha,\beta}B_{\beta,\alpha}=\tau_{0}^{2}I_{H_{0,\alpha}},\,
  B_{\beta,\alpha}B_{\alpha,\beta}=\tau_{0}^{2}I_{H_{0,\beta}}.
\end{equation*}
Последние два равенства равносильны унитарности оператора
$B_{\alpha,\beta}/\tau_{0}$.
\end{example}

\begin{example} \emph{Условие коммутации.} Условие
$P_{\alpha}P_{\beta}=P_{\beta}P_{\alpha}$
равносильно условию
\begin{equation*}
  P_{\alpha}P_{\beta}=P_{\alpha}P_{\beta}P_{\alpha},
\end{equation*}
что, в силу плотности $H_{1}+\ldots+H_{n}$ в $H$, равносильно
\begin{equation*}
  P_{\alpha}P_{\beta}P_i\upharpoonright_{H_{i}}
  =P_{\alpha}P_{\beta}P_{\alpha}P_i\upharpoonright_{H_{i}},\quad
  1\leqslant i\leqslant n.
\end{equation*}
Это условие может быть переписано в виде
$G_{\alpha,\beta}G_{\beta,i}=G_{\alpha,\beta}G_{\beta,\alpha}G_{\alpha,i}$,
$1\leqslant i\leqslant n$, что, в свою очередь, равносильно условию
\begin{equation*}
B_{\alpha,\beta}B_{\beta,i}=B_{\alpha,\beta}B_{\beta,\alpha}B_{\alpha,i},\quad
1\leqslant i\leqslant n.
\end{equation*}
Заметим, что последнее равенство выполнено автоматически при
$i=\alpha$, а при~$i=\beta$ из этого равенства следует, что
$B_{\alpha,\beta}$ является частичной изометрией.
\end{example}

\subsection {Неприводимость системы подпространств $\mathcal{G}(H_{0,1},\ldots,H_{0,n};B)$}\label{sec:irreducibility}

Пусть система подпространств
$S=(H;H_{1},\ldots,H_{n})=\mathcal{G}(H_{0,1},\ldots,H_{0,n};
B)$. Обозначим $P_{k}$ ортопроектор на $H_{k}$, $1\leqslant k\leqslant
n$.

\subsubsection{Спуск оператора}

Пусть
$C:H\to H$ "--- линейный непрерывный оператор, коммутирующий со всеми
$P_{k},\,1\leqslant k\leqslant n$, т.е. для всех $1\leqslant k\leqslant n$
подпространства $H_{k}$ и $H_{k}^{\bot}$ инвариантны относительно $C$.
Для
$1\leqslant k\leqslant n$ определим оператор
$C_k=C\upharpoonright_{H_{k}}:H_{k}\to H_{k}$, тогда
$C_{k}^{*}=C^{*}\upharpoonright_{H_{k}}:H_{k}\to H_{k}$.
Определим оператор
$C_{0,k}:H_{0,k}\to H_{0,k}$ равенством
$C_{0,k}=\Gamma_{k}^{*}\rho_{k}^{-1}C_{k}\rho_{k}\Gamma_{k}$.
Набор операторов
$C_{0,k}$, $1\leqslant k\leqslant n$, будем называть спуском оператора
$C$. Из определения
$C_{0,k}$ следует, что для всех
$1\leqslant k\leqslant n$ диаграмма
\begin{equation*}
\begin{CD}
H_{0,k} @>C_{0,k}>> H_{0,k}\\
@V\Gamma_{k}VV @V\Gamma_{k}VV \\
\widetilde{H}_{k} @>>> \widetilde{H}_{k}\\
@V\rho_{k}VV @V\rho_{k}VV\\
H_{k}@>C_{k}>> H_{k}
\end{CD}
\end{equation*}
коммутативна.

\begin{proposition}
  Для всех $1\leqslant i,j\leqslant n$ справедливо равенство
  \begin{equation}\label{eq:spuskcond}
    B_{i,j}C_{0,j}=C_{0,i}B_{i,j}.
  \end{equation}
\end{proposition}
\begin{proof}
  Рассмотрим произвольные $x\in H_{0,j},\,y\in H_{0,i}$.
  Тогда
\begin{equation*}
\rho \Gamma_{j}C_{0,j}x=C_{j}\rho\Gamma_{j}x=C\rho\Gamma_{j}x.
\end{equation*}
Поскольку  $C_{0,i}^{*}=\Gamma_{i}^{*}\rho_{i}^{-1}C_{i}^{*}\rho_{i}\Gamma_{i}$, то
\begin{equation*}
\rho\Gamma_{i}C_{0,i}^{*}y=C_{i}^{*}\rho\Gamma_{i}y=C^{*}\rho\Gamma_{i}y.
\end{equation*}
Теперь имеем:
  \begin{align*}
    \langle B_{i,j}C_{0,j}x,y\rangle=
    \langle B\Gamma_{j}C_{0,j}x,\Gamma_{i}y\rangle_{0}=
    \langle \rho\Gamma_{j}C_{0,j}x,\rho\Gamma_{i}y\rangle=
    \langle C\rho\Gamma_{j}x,\rho\Gamma_{i}y\rangle=\\
    =\langle \rho\Gamma_{j}x,C^{*}\rho\Gamma_{i}y\rangle=
    \langle \rho\Gamma_{j}x,\rho\Gamma_{i}C_{0,i}^{*}y\rangle=
    \langle B_{i,j}x,C_{0,i}^{*}y\rangle=
    \langle C_{0,i}B_{i,j}x,y\rangle.
  \end{align*}
  Таким образом, равенство~\eqref{eq:spuskcond} доказано.
\end{proof}

Для последовательности индексов
$l=(i(1),i(2),\ldots,i(k))$ определим оператор
\begin{equation*}
B_{l}=B_{i(1),i(2)}\ldots B_{i(k-1),i(k)}:H_{0,i(k)}\to H_{0,i(1)}.
\end{equation*}

\begin{corollary} Для произвольной последовательности индексов
  $l=(i,\ldots,j)$ выполнено равенство
  \begin{equation*}
    C_{0,i}B_{l}=B_{l}C_{0,j}.
  \end{equation*}
\end{corollary}

\subsubsection{Подъем набора операторов}

Пусть
$C_{0,k}:H_{0,k}\to H_{0,k}$, $1\leqslant k\leqslant n$, "--- набор унитарных операторов, причем для произвольных
$i,j$ выполнено
\eqref{eq:spuskcond}. Определим унитарный оператор
$\widetilde{C}=\mathrm{diag}\,(C_{0,1},\ldots,C_{0,n}):\widetilde{H}\to\widetilde{H}$.
Из равенств
\eqref{eq:spuskcond} следует, что
$B=\widetilde{C}^{*}B\widetilde{C}$.

Определим оператор
$C:\rho(\widetilde{H})\to\rho(\widetilde{H})$ равенством
\begin{equation*}
C(x+\widetilde{H}_{0})=\widetilde{C}x+\widetilde{H}_{0},\quad x\in\widetilde{H}.
\end{equation*}
Это определение корректно, так как
$\widetilde{C}(\mathrm{Ker}\,B)=\mathrm{Ker}\,B$.

Для произвольных
$x,y\in\widetilde{H}$ имеем:
\begin{equation*}
\langle x+\widetilde{H}_{0},y+\widetilde{H}_{0}\rangle=
\langle Bx,y\rangle_{0}=
\langle \widetilde{C}^{*}B\widetilde{C}x,y\rangle_{0}=
\langle B\widetilde{C}x,\widetilde{C}y\rangle_{0}=
\langle \widetilde{C}x+\widetilde{H}_{0},\widetilde{C}y+\widetilde{H}_{0}\rangle.
\end{equation*}
Поскольку $\rho(\widetilde{H})$ плотно в $H$, то $C$ "--- линейный
оператор, сохраняющий скалярное произведение, образ которого
$C(\rho(\widetilde{H}))=\rho(\widetilde{H})$ плотен в $H$. Поэтому $C$
продолжается единственным образом (по непрерывности) до унитарного
оператора (который мы также обозначим~$C$) $C:H\to H$. Оператор $C$
называется подъемом набора операторов $C_{0,k}$, $1\leqslant
k\leqslant n$.  Из определения $C$ следует, что $C(H_{k})=H_{k}$ ,
$1\leqslant k\leqslant n$. Из унитарности $C$ получим
$C(H_{k}^{\bot})=H_{k}^{\bot}$, $1\leqslant k\leqslant n$.

Очевидно, спуск оператора
$C$ совпадает с набором $C_{0,k}$, $1\leqslant k\leqslant n$.

\subsubsection{Неприводимость системы подпространств $S$}

Для последовательности индексов (пути)
$l=(i(1),\ldots,i(k-1),i(k))$ определим путь
$l^{*}=(i(k),i(k-1),\ldots,i(1))$. Ясно, что
$B_{l}^{*}=B_{l^{*}}$. Для двух путей
$l,l'$ таких, что конец
$l$ совпадает с началом $l'$,
$l=(i(1),\ldots,i(k-1),i(k))$,
$l'=(i(k),i(k+1),\ldots,i(m))$, определим путь
$ll'=(i(1),\ldots,i(k-1),i(k),i(k+1),\ldots,i(m))$.
Далее
$\alpha$ обозначает натуральное число,
$1\leqslant\alpha\leqslant n$. Обозначим через
$\mathcal{L}_{\alpha}$ множество путей
$l=(\alpha,\ldots,\alpha)$ с началом и концом
$\alpha$.
Отметим, что множество операторов
$B_{l},\,l\in\mathcal{L}_{\alpha}$, является $\ast\text{-}$\hspace{0pt}множеством, т.е.
если оператор $A$ принадлежит этому множеству, то оператор
$A^{*}$ также ему принадлежит.

\begin{proposition}\label{P:irreducibility}
Пусть $\alpha$ таково, что для произвольного
$k,\,1\leqslant k\leqslant n$, существует путь
$l=(\alpha,\ldots,k)$, для которого оператор
$B_{l}$ обратим. Тогда следующие утверждения эквивалентны:
\begin{enumerate}
\item
система подпространств
$S=(H;H_{1},\ldots,H_{n})$ неприводима,
\item
множество операторов
$B_{l},\,l\in\mathcal{L}_{\alpha}$, неприводимо.
\end{enumerate}
\end{proposition}
\begin{proof}
$(1)\Rightarrow(2)$
Предположим противное. Тогда существует оператор
$C_{0,\alpha}:H_{0,\alpha}\to H_{0,\alpha}$, отличный от $\lambda I_{H_{0,\alpha}},\,\lambda\in\mathbb{C}$, такой, что
$C_{0,\alpha}B_{l}=B_{l}C_{0,\alpha}$ для всякого пути
$l\in\mathcal{L}_{\alpha}$.
Поскольку множество
$B_{l},\,l\in\mathcal{L}_{\alpha}$, является
$\ast\text{-}$\hspace{0pt}множеством, то оператор
$C_{0,\alpha}$ можно выбрать унитарным.
Из условия утверждения следует, что для каждого
$k$ существует путь $l(k)=(k,\ldots,\alpha)$, для которого
$B_{l(k)}$ обратим.
Определим оператор
\begin{equation*}
C_{0,k}=B_{l(k)}C_{0,\alpha}B_{l(k)}^{-1}:H_{0,k}\to H_{0,k},\quad 1\leqslant k\leqslant n.
\end{equation*}
Отметим, что для $k=\alpha$ определение корректно, так как операторы $C_{0,\alpha}$ и $B_{l(\alpha)}$
коммутируют.

Покажем, что
$C_{0,k}$ унитарен.
Ясно, что
$C_{0,k}^{*}=(B_{l(k)}^{*})^{-1}C_{0,\alpha}^{*}B_{l(k)}^{*}$.
Таким образом,
$C_{0,k}$ обратим, и
\begin{equation*}
C_{0,k}^{*}C_{0,k}=(B_{l(k)}^{*})^{-1}C_{0,\alpha}^{*}B_{l(k)}^{*}B_{l(k)}C_{0,\alpha}B_{l(k)}^{-1}=
(B_{l(k)}^{*})^{-1}C_{0,\alpha}^{*}C_{0,\alpha}B_{l(k)}^{*}B_{l(k)}B_{l(k)}^{-1}=I_{H_{0,k}},
\end{equation*}
откуда следует унитарность $C_{0,k}$.

Покажем, что для произвольных
$i,j$
$C_{0,i}B_{i,j}=B_{i,j}C_{0,j}$. Это равенство равносильно равенству
$B_{l(i)}C_{0,\alpha}B_{l(i)}^{-1}B_{i,j}=B_{i,j}B_{l(j)}C_{0,\alpha}B_{l(j)}^{-1}$, что равносильно
\begin{equation}\label{E:irreduc}
B_{l(i)}^{*}B_{l(i)}C_{0,\alpha}B_{l(i)}^{-1}B_{i,j}=B_{l(i)}^{*}B_{i,j}B_{l(j)}C_{0,\alpha}B_{l(j)}^{-1}.
\end{equation}
Поскольку путь
$l(i)^{*}l(i)\in\mathcal{L}_{\alpha}$, оператор
$B_{l(i)}^{*}B_{l(i)}$ коммутирует с
$C_{0,\alpha}$. Поскольку путь
$l(i)^{*}(i,j)l(j)\in\mathcal{L}_{0,\alpha}$, оператор
$B_{l(i)}^{*}B_{i,j}B_{l(j)}$ коммутирует с
$C_{0,\alpha}$. Поэтому обе части равенства \eqref{E:irreduc} равны
$C_{0,\alpha}B_{l(i)}^{*}B_{i,j}$, и, таким образом, равенство
\eqref{E:irreduc} верно.

Поднимем семью унитарных операторов
$C_{0,k},\,1\leqslant k\leqslant n$, до унитарного оператора
$C:H\to H$. Тогда
$CP_{k}=P_{k}C,\,1\leqslant k\leqslant n$. Поскольку система подпространств
$S$ неприводима, то для некоторого
$\lambda\in\mathbb{C}$
$C=\lambda I_{H}$. Поэтому
$C_{0,\alpha}=\lambda I_{H_{0,\alpha}}$, получили противоречие.

$(2)\Rightarrow (1)$
Предположим, линейный непрерывный оператор
$C:H\to H$ коммутирует со всеми
$P_{k},\,1\leqslant k\leqslant n$. Пусть набор операторов
$C_{0,k},\,1\leqslant k\leqslant n$, получен спуском
$C$. Поскольку
$C_{0,\alpha}B_{l}=B_{l}C_{0,\alpha}$ для всякого
$l\in\mathcal{L}_{\alpha}$, то существует
$\lambda\in\mathbb{C}$, для которого
$C_{0,\alpha}=\lambda I_{H_{0,\alpha}}$. Рассмотрим произвольное
$1\leqslant k\leqslant n$ и выберем путь
$l=(\alpha,\ldots,k)$, для которого
$B_{l}$ обратим.
Поскольку
$C_{0,\alpha}B_{l}=B_{l}C_{0,k}$, то
$B_{l}(C_{0,k}-\lambda I_{H_{0,k}})=0$, откуда
$C_{0,k}=\lambda I_{H_{0,k}}$. Из доказанного следует, что для
$x\in\rho(\widetilde{H})$
$Cx=\lambda x$. Из плотности
$\rho(\widetilde{H})$ в $H$ следует, что
$C=\lambda I_{H}$. Это доказывает неприводимость системы подпространств
$S$.
\end{proof}

\section{Описание систем, удовлетворяющих условиям (Ang), (Com), (Ort)}
\label{S:staralgebras}

В этом разделе мы
\begin{enumerate}
\item получим описание всех систем подпространств, удовлетворяющих
  условиям (Ang), (Com), (Ort);

  \item получим описание всех \emph{неприводимых} унитарно
    неэквивалентных систем, удовлетворяющих условиям (Ang), (Com), (Ort);

  \item в качестве примера, приведем описание всех неприводимых
    унитарно неэквивалентных систем подпространств, удовлетворяющих
    (Ang), (Com), (Ort), в случае $m=3$ и $r=1$.
\end{enumerate}

Сначала покажем, что без ограничения общности можно считать, что
$\tau_{\{1,2k\}}=\tau_{\{1,2k+1\}}$ для всех
$1\leqslant k\leqslant m$.

\begin{lemma}\label{lemma:eq}
  Пусть $M,M_{1},M_{2}$ "--- подпространства $H$. Предположим, что
  выполнены следующие условия:
  \begin{enumerate}
  \item
    подпространства $M,M_{i}$ расположены друг относительно друга под углом
    $\varphi_{i}\in[0,\pi/2)$, $i=1,2$,
    \item
      ортопроекторы на подпространства
      $M_{1},M_{2}$ коммутируют.
  \end{enumerate}
  Тогда если
  $\varphi_{1}\neq\varphi_{2}$, то
  подпространства $M_{1},M_{2}$ ортогональны.
\end{lemma}
\begin{proof}
  Обозначим
  $Q$, $Q_{1}$, $Q_{2}$ ортопроекторы на
  $M$, $M_{1}$, $M_{2}$ соответственно. Определим
  $\mu_{k}=\cos\varphi_{k}$ для
  $k=1,2$. Имеем:
  \begin{align*}
    &Q_{1}Q_{2}=Q_{1}Q_{2}Q_{1}=\frac{1}{\mu_{2}^{2}}Q_{1}Q_{2}QQ_{2}Q_{1}=\\
    &=\frac{1}{\mu_{2}^{2}}Q_{2}Q_{1}QQ_{1}Q_{2}=\frac{\mu_{1}^{2}}{\mu_{2}^{2}}Q_{2}Q_{1}Q_{2}=\frac{\mu_{1}^{2}}{\mu_{2}^{2}}Q_{1}Q_{2},
  \end{align*}
  откуда
  $Q_{1}Q_{2}=0$.
\end{proof}

Итак, уменьшив $m$ (если это необходимо), можно считать, что
$\tau_{\{1,2k\}}=\tau_{\{1,2k+1\}}$ для всех $1\leqslant k\leqslant
m$.  Определим
\begin{equation*}
  \tau_k=\begin{cases}
  \tau_{\{1,2k\}}=\tau_{\{1, 2k+1\}}, & 1\leqslant k\leqslant m;\\
  \tau_{\{1,k+m+1\}}, & m+1\leqslant k\leqslant m+r.
  \end{cases}
\end{equation*}

\subsection{Описание всех систем подпространств
$S=(H;H_{1},\ldots,H_{N+1})$, удовлетворяющих условиям
(Ang), (Com), (Ort)}\label{SS:all}
Прежде всего сделаем несколько очевидных замечаний.

\textbf{1.} Нулевая система $S=(H;0,\ldots,0)$ удовлетворяет всем
нужным условиям. В дальнейшем будем рассматривать ненулевые системы
подпространств.  Отметим, что если система
$S=(H;H_{1},\ldots,H_{N+1})$ удовлетворяет условиям (Ang), и для
некоторого $k$ подпространство $H_{k}=0$, то, как легко видеть,
$H_{1}=\ldots=H_{N+1}=0$.  Таким образом, если система ненулевая, то
$H_{k}\ne0$, $1\leqslant k\leqslant N+1$.

\textbf{2.} Пусть $S=(H;H_{1},\ldots,H_{N+1})$ "--- ненулевая система
подпространств, удовлетворяющая (Ang), (Com), (Ort). Предположим,
$H_{1}+\ldots+H_{N+1}$ не плотно в $H$. Определим системы
\begin{equation*}
  S'=(H';H_{1},\ldots,H_{N+1})
  \quad\text{и}\quad
  S''=(H\ominus H';0,\ldots,0),
\end{equation*}
где $H'=\overline{H_{1}+\ldots+H_{N+1}}$.  Тогда $S=S'\oplus S''$,
$S'$ удовлетворяет (Ang), (Com), (Ort) и $S''$ является нулевой
системой. Таким образом, чтобы описать все системы подпространств
удовлетворяющие условиям (Ang), (Com), (Ort) достаточно описать
системы, для которых сумма~$H_{1}+\ldots+H_{N+1}$ плотна в $H$.

\textbf{3.} Предположим теперь, что $S=(H;H_{1},\ldots,H_{N+1})$ "---
система подпространств, удовлетворяющая (Ang), (Com), (Ort), такая,
что $H_{k}\neq 0$, $1\leqslant k\leqslant N+1$, и сумма
$H_{1}+\ldots+H_{N+1}$ плотна в $H$. Пусть $G=(G_{i,j},\,1\leqslant
i,j\leqslant N+1)$ "--- оператор Грама системы $S$. Тогда $S$ унитарно
эквивалентна системе $\mathcal{G}(H_{1},\ldots,H_{N+1};G)$. Поскольку
$H_{1},H_{k}$ расположены друг относительно друга под углом
$\theta_{\{1,k\}}$, то оператор $G_{1,k}/\tau_{\{1,k\}}$ унитарный.
Определим унитарные операторы $U_{0,k}:H_{1}\to H_{k}$, $1\leqslant
k\leqslant N+1$, формулами
\begin{equation*}
  U_{0,1}=I_{H_{1}},\qquad
  U_{0,k} = G_{1,k}^*/\tau_{\{1,k\}},\; 2\leqslant k\leqslant N+1
\end{equation*}
Положим $B_{i,j}=U_{0,i}^{*}G_{i,j}U_{0,j}$, $1\leqslant i,j\leqslant
N+1$, тогда $B_{1,k}=\tau_{\{1,k\}}I_{H_{1}}$, $2\leqslant k\leqslant
N+1$. Определим оператор
$B:\oplus_{k=1}^{N+1}H_{1}\to\oplus_{k=1}^{N+1}H_{1}$ блочным
разложением $B=(B_{i,j})$. Из утверждения \ref{P:unitequiv} следует,
что $\mathcal{G}(H_{1},\ldots,H_{N+1};G)$ унитарно эквивалентна
$\mathcal{G}(H_{1},\ldots,H_{1};B)$, а поэтому $S$ унитарно
эквивалентна $\mathcal{G}(H_{1},\ldots,H_{1};B)$.

Пусть теперь $S=(H;H_{1},\ldots,H_{N+1})=\mathcal{G}(H_{0},\ldots,H_{0};B)$
для некоторого гильбертова пространства $H_{0}$ и оператора
$B:\oplus_{k=1}^{N+1}H_{0}\to\oplus_{k=1}^{N+1}H_{0}$ такого, что
$B_{1,k}=\tau_{\{1,k\}}I_{H_{0}}$, $2\leqslant k\leqslant
N+1$. Выясним, каким условиям должен удовлетворять оператор~$B$, чтобы
система подпространств~$S$ удовлетворяла (Ang), (Com), (Ort). Для
этого воспользуемся результатами раздела \ref{SS:properties}.

Условие~(Ang) равносильно унитарности операторов
$B_{1,k}/\tau_{\{1,k\}}$, $2\leqslant k\leqslant N+1$.  Поскольку
$B_{1,k}=\tau_{\{1,k\}}I_{H_{0}}$, это условие выполнено.

Условие~(Ort) равносильно $B_{i,j}=0$ для всех пар различных $i,j$,
таких, что $\{i,j\}\notin E\cup E^c_m$.

Рассмотрим условие~(Com).  Условие $P_{2k}P_{2k+1}=P_{2k+1}P_{2k}$,
$1\leqslant k\leqslant m$, равносильно
\begin{equation}\label{E:com}
  B_{2k,2k+1}B_{2k+1,i}=B_{2k,2k+1}B_{2k+1,2k}B_{2k,i}\quad 1\leqslant i\leqslant N+1.
\end{equation}

При $i=1$ имеем равенство $\tau_{k}B_{2k,2k+1}=\tau_{k}B_{2k,2k+1}B_{2k+1,2k}$,
т.е.  $B_{2k,2k+1}$ "--- ортопроектор.

При $i=2k$ равенство \eqref{E:com} выполнено автоматически.

При $i=2k+1$ имеем $B_{2k,2k+1}=B_{2k,2k+1}B_{2k+1,2k}B_{2k,2k+1}$.
Это условие выполнено, поскольку $B_{2k,2k+1}$ "--- ортопроектор.

При $i\neq 1,2k,2k+1$ обе части равенства \eqref{E:com} равны $0$,
следовательно, оно выполнено.

Таким образом, оператор $B$ имеет вид
\begin{equation*}
B=
\begin{pmatrix}
          I& \tau_1I& \tau_1I&  \ldots& \tau_{m}I& \tau_m I& \tau_{m+1}I& \ldots& \tau_{m+r}I\\
    \tau_1I&          I&   Q_{1}& \ldots&              0&             0&                  0& \ldots&          0\\
    \tau_1I&  Q_{1}&           I& \ldots&              0&             0&                  0& \ldots&          0\\
     \vdots&  \vdots&  \vdots& \ddots&             0&             0&                  0& \ldots&          0\\
    \tau_mI&          0&         0&         0&             I&     Q_{m}&                  0& \ldots&          0\\
    \tau_mI&          0&         0&         0&    Q_{m}&              I&                  0& \ldots&          0\\
 \tau_{m+1}I&         0&         0&         0&             0&             0&                   I& \ldots&          0\\
         \vdots& \vdots& \vdots& \vdots&    \vdots&      \vdots&          \vdots& \ddots&         0\\
     \tau_{m+r}I&        0&         0&         0&            0&              0&                  0&  \ldots&          I
\end{pmatrix},
\end{equation*}
для некоторого набора ортопроекторов~$Q_k$, $1\leqslant k\leqslant m$, в дальнейшем будем обозначать его $B(Q_{1},\ldots,Q_{m})$.

Неотрицательность оператора $B(Q_1,\ldots,Q_{m})$, где~$Q_k$ "---
некоторый набор ортопроекторов в гильбертовом пространстве $H_0$,
является необходимым условием того, что для данного набора
ортопроекторов можно применить
$G\text{-}$\hspace{0pt}конструкцию. Для изучения вопроса, когда
оператор такого вида неотрицателен, нам понадобиться следующая лемма.

\begin{lemma}\label{lemma:B_min}
  Пусть $K$ "--- гильбертово пространство,
  $A_{1},\ldots,A_{n}$ "--- неотрицательные обратимые операторы в
  $K$.
  Пусть $y\in K$ и $\mu_k>0,\,1\leqslant k\leqslant n$.
  Если $u_{k}\in K$, $1\leqslant k\leqslant n$, и
  $\sum_{k=1}^{n}\mu_{k}u_{k}=y$, то
  \begin{equation*}
  \sum_{k=1}^{n}\langle\,A_{k}u_k,u_k\,\rangle \geqslant
    \left\langle\,\left(\sum_{j=1}^n\mu_j^2A_j^{-1}\right)^{-1}y, y\,\right\rangle,
  \end{equation*}
  причем равенство достигается тогда и только тогда, когда
  \begin{equation*}
  u_{k}=\mu_{k}A_{k}^{-1}(\sum_{j=1}^{n}\mu_{j}^{2}A_{j}^{-1})^{-1}y,\quad 1\leqslant k\leqslant n.
  \end{equation*}
\end{lemma}

\begin{proof}
  Обозначим $F(u_1,\dots,u_n)=\sum_{k=1}^n\langle\,
  A_ku_k,u_k\,\rangle.$ Рассмотрим
 \begin{equation*}
 F(u_{1}+h_{1},\ldots,u_{n}+h_{n})-F(u_{1},\ldots,u_{n})=\sum_{k=1}^{n}
 (\langle A_{k}u_{k},h_{k}\rangle +\langle A_{k}h_{k},u_{k}\rangle +\langle A_{k}h_{k},h_{k}\rangle).
 \end{equation*}
 Если
  $u_{1},\ldots,u_{n}\in K$ таковы, что
  \begin{equation}\label{E:lemmacond}
  \sum_{k=1}^{n}\mu_{k}u_{k}=y,\quad \mu_{1}^{-1}A_{1}u_{1}=\ldots=\mu_{n}^{-1}A_{n}u_{n},
  \end{equation}
  то для произвольных
  $h_{1},\ldots,h_{n}\in K$, таких, что $\sum_{k=1}^{n}\mu_{k}h_{k}=0$, имеем
  \begin{align*}
  \sum_{k=1}^{n}\langle A_{k}u_{k},h_{k}\rangle=\sum_{k=1}^{n}\langle\mu_{k}^{-1}A_{k}u_{k},\mu_{k}h_{k}\rangle=0,\\
  \sum_{k=1}^{n}\langle A_{k}h_{k},u_{k}\rangle=\sum_{k=1}^{n}\langle \mu_{k}h_{k},\mu_{k}^{-1}A_{k}u_{k}\rangle=0,
  \end{align*}
  и, следовательно,
  \begin{equation*}
  F(u_{1}+h_{1},\ldots,u_{n}+h_{n})-F(u_{1},\ldots,u_{n})=\sum_{k=1}^{n}\langle A_{k}h_{k},h_{k}\rangle.
  \end{equation*}
  Отсюда
  $F(u_{1}+h_{1},\ldots,u_{n}+h_{n})\geqslant F(u_{1},\ldots,u_{n})$ и равенство достигается тогда и только тогда, когда
  $h_{1}=\ldots=h_{n}=0$, т.е.
  $F(v_{1},\ldots,v_{n})\geqslant F(u_{1},\ldots,u_{n})$ для произвольных
  $v_{1},\ldots,v_{n}\in K$, таких, что
  $\sum_{k=1}^{n}\mu_{k}v_{k}=y$, и равенство достигается тогда и только тогда, когда
  $v_{k}=u_{k}$ для всех $1\leqslant k\leqslant n$.

  Найдем $u_{1},\ldots,u_{n}$, удовлетворяющие
  \eqref{E:lemmacond}. Пусть $\mu_{k}^{-1}A_{k}u_{k}=x$, $1\leqslant
  k\leqslant n$. Тогда $u_{k}=\mu_{k}A_{k}^{-1}x$. Имеем:
  $\sum_{k=1}^{n}\mu_{k}u_{k}=\sum_{k=1}^{n}\mu_{k}^{2}A_{k}^{-1}x=y$,
  откуда $x=(\sum_{k=1}^{n}\mu_{k}^{2}A_{k}^{-1})^{-1}y$. Таким
  образом,
  $u_{k}=\mu_{k}A_{k}^{-1}(\sum_{j=1}^{n}\mu_{j}^{2}A_{j}^{-1})^{-1}y$,
  $1\leqslant k\leqslant n$, и
  \begin{equation*}
    F(u_{1},\ldots,u_{n})=
    \sum_{k=1}^{n}
    \left\langle\,\mu_{k}(\sum_{j=1}^{n}\mu_j^2A_{j}^{-1})^{-1}y,\mu_{k}A_{k}^{-1}(\sum_{j=1}^{n}\mu_j^2A_{j}^{-1})^{-1}y\,\right\rangle=
    \langle\,(\sum_{j=1}^{n}\mu_j^2A_{j}^{-1})^{-1}y,y\,\rangle.
  \end{equation*}
\end{proof}

Предыдущая лемма позволяет доказать следующее утверждение.

\begin{proposition}\label{P:nonnegB}
  Пусть
  \begin{equation*}
    \xi(\tau)=1-\sum_{k=1}^{m+r}\tau_k^2.
  \end{equation*}
  Оператор $B=B(Q_1,\ldots, Q_m)$ неотрицателен тогда и только тогда,
  когда
  \begin{equation}\label{E:sumort}
    \sum_{k=1}^{m}\tau_k^2R_{k}\leqslant \xi(\tau) I,
  \end{equation}
  где $R_k=I-Q_k$, $1\leqslant k\leqslant m$.
\end{proposition}
\begin{proof}
  Запишем условие неотрицательности
  \begin{equation*}
    \langle Bx,x\rangle\geqslant 0,
    \qquad x=(z,x_{1},y_{1},\ldots,x_{m},y_{m},v_1,\ldots, v_r)\in \oplus_{k=1}^{N+1}H_{0}.
  \end{equation*}
  Обозначим
  \begin{align*}
    z_0&=z_0(x_1,y_1,\ldots, x_m, y_m, v_1, \ldots, v_r) = \sum_{k=1}^{m}\tau_k(x_{k}+y_{k})+\sum_{k=1}^{r}\tau_{k+m} v_{k},\\
    B_0&=B_0(x_{1},y_{1},\ldots,x_{m},y_{m}) = \sum_{k=1}^{m}\left(\|x_{k}\|^{2}+\|y_{k}\|^{2}+2\mathrm{Re}\,\langle\,Q_{k}x_{k},y_{k}\,\rangle\right),
  \end{align*}
  тогда имеем условие:
  \begin{equation*}
    \langle\,Bx,x\,\rangle=
    \|z\|^{2}+2\mathrm{Re}\,\langle\,z,z_0\,\rangle+B_{0}+
    \sum_{k=1}^{r}\|v_k\|^2
    \geqslant 0.
  \end{equation*}
  Это условие эквивалентно условию
  \begin{equation*}
    \|z + z_0\|^{2} - \|z_0\|^2+ B_{0}+\sum_{k=1}^{r}\|v_k\|^2\geqslant 0.
  \end{equation*}
  Левая часть этого неравенства достигает наименьшего значения по $z$ при
  $z=-z_0$.  Таким образом, оператор $B$ неотрицателен тогда и только тогда, когда
  \begin{equation}\label{eq:without_z}
    -\|z_{0}\|^{2}+B_{0}+\sum_{k=1}^{r}\|v_k\|^2 \geqslant 0.
  \end{equation}
  Пусть
  \begin{equation}\label{E:forkerB}
  z_k=\dfrac{x_k+y_k}{2}, \delta_k=\dfrac{x_k-y_k}{2},
  \end{equation}
  тогда $x_{k}=z_{k}+\delta_{k}$ и $y_{k}=z_{k}-\delta_{k}$.
  Ясно, что
  \begin{equation}\label{E:z_0}
  z_{0}=2\sum_{k=1}^{m}\tau_{k}z_{k}+\sum_{k=1}^{r}\tau_{k+m}v_{r}.
  \end{equation}
  Кроме того,
  \begin{multline}\label{E:B_0}
      \quad\|x_{k}\|^{2}+\|y_{k}\|^{2}+2\mathrm{Re}\langle\,Q_{k}x_{k},y_{k}\,\rangle
      =\|R_{k}x_{k}\|^{2}+\|R_{k}y_{k}\|^{2}+\|Q_{k}(x_{k}+y_{k})\|^{2}\\
      =2\|R_{k}z_{k}\|^{2}+2\|R_{k}\delta_{k}\|^{2}+4\|Q_{k}z_{k}\|^{2}
      =\langle\,(2I+2Q_{k})z_{k},z_{k}\,\rangle+2\|R_{k}\delta_{k}\|^{2}.
  \end{multline}
  Используя равенства \eqref{E:z_0}, \eqref{E:B_0}, перепишем
  \eqref{eq:without_z} в виде
  \begin{equation*}
    \sum_{k=1}^{m}\,\langle2(I+Q_{k})z_{k},z_{k}\,\rangle+\sum_{k=1}^r\|v_k\|^2-
    \|2\sum_{k=1}^m \tau_kz_{k}+\sum_{k=1}^r\tau_{k+m}v_k\|^{2}
    +2\sum_{k=1}^{m}\|R_{k}\delta_{k}\|^{2}\geqslant 0.
  \end{equation*}
  Полученное неравенство выполнено для всех $z_{1},\ldots,z_{m}$,
  $\delta_{1},\ldots,\delta_{m}$, $v_{1},\ldots,v_{r}$, тогда и только
  тогда, когда неравенство
  \begin{equation}\label{E:nonneg1}
    \sum_{k=1}^{m}\,\langle\,2(I+Q_{k})z_{k},z_{k}\,\rangle+\sum_{k=1}^r\|v_{k}\|^{2}\geqslant
    \|2\sum_{k=1}^m\tau_kz_{k}+\sum_{k=1}^r\tau_{k+m}v_k\|^{2}.
  \end{equation}
  выполнено для всех $z_{1},\ldots,z_{m}$, $v_{1},\ldots,v_{r}$.

  Зафиксируем
  $2\sum_{k=1}^{m}\tau_{k}z_{k}+\sum_{k=1}^{r}\tau_{k+m}v_{k}=y$
  и обозначим
  \begin{equation*}
    A_{k}=\begin{cases}
    2(I+Q_{k}), &1\leqslant k\leqslant m\\
    I, &m+1\leqslant k\leqslant m+r,
    \end{cases}\qquad
    \mu_k = \begin{cases}
      2\tau_{k},  &1\leqslant k\leqslant m\\
      \tau_k, &m+1\leqslant k\leqslant m+r
    \end{cases}
  \end{equation*}
  Из леммы~\ref{lemma:B_min} следует, что наименьшее значение левой части ~\eqref{E:nonneg1}
  (при фиксированном~$y$) равно
  $\langle\,(\sum_{k=1}^{m+r}\mu_{k}^{2}A_{k}^{-1})^{-1}y,y\,\rangle$.
  Таким образом, оператор $B$ неотрицателен тогда и только тогда, когда
  \begin{equation}\label{eq:nonneg2}
    \left(\sum_{k=1}^{m+r}\mu_k^2A_{k}^{-1}\right)^{-1}\geqslant I,\quad \text{т.е.}\quad
    \sum_{k=1}^{m+r}\mu_{k}^{2}A_{k}^{-1}\leqslant I.
  \end{equation}
  Поскольку
  $(2(I+Q_{k}))^{-1}=\dfrac{1}{4}(I+R_k)$, то
  \begin{equation*}
  \sum_{k=1}^{m+r}\mu_{k}^{2}A_{k}^{-1}=
  \sum_{k=1}^{m}4\tau_{k}^{2}\cdot\frac{1}{4}(I+R_{k})+\sum_{k=m+1}^{m+r}\tau_{k}^{2}I=
  \sum_{k=1}^{m+r}\tau_{k}^{2}I+\sum_{k=1}^{m}\tau_{k}^{2}R_{k}.
  \end{equation*}
  Поэтому \eqref{eq:nonneg2} можно переписать в виде
  \begin{equation*}
    \sum_{k=1}^{m+r}\tau_{k}^{2}I+\sum_{k=1}^{m}\tau_{k}^{2}R_{k}\leqslant I,\quad \text{т.е.}\quad
    \sum_{k=1}^{m}\tau_{k}^{2}R_{k}\leqslant \left(1-\sum_{k=1}^{m+r}\tau_{k}^{2}\right)I.
  \end{equation*}
\end{proof}

Следующее утверждение дает описание $\mathrm{Ker}\,B$ и следует из доказательства утверждения
\ref{P:nonnegB} и леммы
\ref{lemma:B_min}. Напомним, что
$z_{k},\delta_{k}$, $1\leqslant k\leqslant m$, определены формулами \eqref{E:forkerB}.

\begin{proposition}\label{P:KerB}
  Пусть выполнено \eqref{E:sumort}. Элемент
  $x=(z,x_{1},y_{1},\ldots,x_{m},y_{m},v_{1},\ldots,v_{r})$
  принадлежит $\mathrm{Ker}\,B$ тогда и только тогда, когда
  \begin{enumerate}
  \item
    $z=-(2\sum_{k=1}^{m}\tau_{k}z_{k}+\sum_{k=1}^{r}\tau_{k+m}v_{k})$;
  \item
    $\delta_{k}\in\mathrm{Im}\,Q_{k}=\mathrm{Ker}\,R_k$ для всех
    $1\leqslant k\leqslant m$;
  \item
    $z_{k}=\frac{1}{2}\tau_{k}(I+R_{k})y$, $1\leqslant k\leqslant m$, и
    $v_{k}=\tau_{k+m}y$, $1\leqslant k\leqslant r$, где
    $y\in\mathrm{Ker}(\xi(\tau)I-\sum_{k=1}^{m}\tau_{k}^{2}R_{k})$.
  \end{enumerate}
\end{proposition}

\begin{corollary}
Пусть выполнено \eqref{E:sumort} и
$H_{0}$ конечномерно. Тогда
\begin{equation}\label{E:KerB}
\dim\mathrm{Ker}\,B=\sum_{k=1}^{m}\dim\mathrm{Im}\,Q_{k}+\dim\mathrm{Ker}(\xi(\tau)I-\sum_{k=1}^{m}\tau_{k}^{2}R_{k}).
\end{equation}
\end{corollary}

Критерий унитарной эквивалентности (утверждение~\ref{P:unitequiv}) для
рассматриваемых систем можно сформулировать в терминах ортопроекторов
$Q_1,\dots, Q_m$.
\begin{proposition}\label{P:unitequiv_Q}
  Системы подпространств
  \begin{equation*}
    \mathcal{G}(H_{0},\ldots,H_{0};B(Q_{1},\ldots,Q_{m}))
    \quad \text{и}\quad
    \mathcal{G}(H'_{0},\ldots,H'_{0};B(Q_{1}',\ldots,Q_{m}'))
  \end{equation*}
  унитарно эквивалентны тогда и только тогда, когда наборы
  ортопроекторов $Q_{k},\,1\leqslant k\leqslant m$, и
  $Q_{k}',\,1\leqslant k\leqslant m$, унитарно эквивалентны;
\end{proposition}

Таким образом, используя $G$-конструкцию систем подпространств, мы
установили взаимно однозначное соответствие между ненулевыми системами
подпространств $S$, удовлетворяющими (Ang), (Com), (Ort), такими, что
$H_{1}+\ldots+H_{N+1}$ плотно в $H$, и наборами $m$ ортопроекторов
$R_{1},\ldots,R_{m}$ в некотором гильбертовом пространстве $H_{0}$,
удовлетворяющими неравенству~\eqref{E:sumort} (при этом системы
подпространств и наборы операторов рассматриваются с точностью до
унитарной эквивалентности).

Отметим, что необходимым условием выполнения \eqref{E:sumort} является
$\xi(\tau)\geqslant 0$. Поэтому если $\xi(\tau)<0$, то не существует
ненулевой системы подпространств $S$, удовлетворяющей (Ang), (Com),
(Ort).  В дальнейшем мы предполагаем, что $\xi(\tau)\geqslant 0$.

\subsection{Описание всех неприводимых унитарно неэквивалентных систем
$S$, удовлетворяющих (Ang), (Com), (Ort)}\label{SS:irred}

Перед тем как перейти к описанию всех неприводимых систем подпространств, отметим, что
\begin{itemize}
\item с точностью до унитарной эквивалентности существует единственная
  нулевая неприводимая система подпространств $S=(\mathbb{C}^{1};0,\ldots,0)$;
\item для любой ненулевой неприводимой системы
  подпространств~$S=(H;H_{1},\ldots,H_{N+1})$
  $H_{1}+\ldots+H_{N+1}$ плотно в $H$.
\end{itemize}
Кроме того, критерий неприводимости
(утверждение~\ref{P:irreducibility}) рассматриваемых систем в терминах
ортопроекторов~$Q_1,\dots,Q_m$ может быть сформулирован в следующем
виде.
\begin{proposition}\label{P:irred_Q}
  Система подпространств
  $\mathcal{G}(H_{0},\ldots,H_{0};B(Q_{1},\ldots,Q_{m}))$ неприводима
  тогда и только тогда, когда набор ортопроекторов $Q_{k},\,1\leqslant
  k\leqslant m$, неприводим.
\end{proposition}

Поэтому с точностью до унитарной эквивалентности все ненулевые
неприводимые системы подпространств, удовлетворяющие условиям (Ang),
(Com), (Ort) имеют вид
$S=\mathcal{G}(H_{0},\ldots,H_{0};B(Q_{1},\ldots,Q_{m}))$, где $H_{0}$
"--- гильбертово пространство, $Q_{1},\ldots,Q_{m}$ "--- неприводимая
семья ортопроекторов в $H_{0}$, такая, что для
ортопроекторов~$R_{k}=I-Q_{k}$, $1\leqslant k\leqslant m$, выполнено
неравенство~\eqref{E:sumort}.

Учитывая, что системы $S$ и $S'$ унитарно эквивалентны тогда и только
тогда, когда унитарно эквивалентны наборы ортопроекторов
$Q_1,\dots,Q_m$ и $Q'_1,\dots,Q'_m$, задача описания всех
неприводимых унитарно неэквивалентных ненулевых систем подпространств
$S$, удовлетворяющих (Ang), (Com), (Ort), эквивалентна задаче об
описании неприводимых унитарно неэквивалентных наборов ортопроекторов
$R_{1},\ldots,R_{m}$, удовлетворяющих \eqref{E:sumort}.

Если $\xi(\tau)=0$, то $R_{1}=\ldots=R_{m}=0$, т.е
$Q_{1}=\ldots=Q_{m}=I$. Поскольку набор $R_{1},\ldots,R_{m}$
неприводим, то $H_{0}=\mathbb{C}^{1}$. Используя формулу
\eqref{E:KerB}, получим $\dim\mathrm{Ker}\,B=m+1$. Непосредственно из
определения $G\text{-}$\hspace{0pt}конструкции системы подпространств
следует, что $\dim H=(N+1)-(m+1)=m+r$, $\dim H_{k}=\dim H_{0}=1$ для
всех $1\leqslant k\leqslant N+1$. Поэтому обобщенная размерность
системы $S$ равна $(m+r;1)$.

Рассмотрим случай $\xi(\tau)>0$.

Множество индексов $M=\{1,2,\dots,m\}$ разобьем на три части
\begin{equation*}
  M_{l}=\{k\in M\,|\, \tau_{k}^2< \xi(\tau) \},
  \quad
  M_{e} = \{k\in M\,|\, \tau_k^2=\xi(\tau) \},
  \quad
  M_{g}=\{k\in M\,|\, \tau_k^2>\xi(\tau) \}.
\end{equation*}
Без ограничения общности, будем считать, что для любых $k_1\in M_l$,
$k_2\in M_e$ и $k_3\in M_g$, выполнены неравенства $k_1<k_2<k_3$.

Если $i\in M_g$, то, очевидно, $R_{i}=0$.

Если $i\in M_e$, то $R_{i}R_{j}=0$ для всех $j\neq i$. Поскольку набор
$R_{1},\ldots,R_{m}$ неприводим, то $R_{i}=0$ или $R_{i}=I$.
Если $R_{i}=I$, то $R_{j}=0$ для всех $j\neq i$.

Предположим, для некоторого $i\in M_e$~ $R_{i}=I$ и $R_{j}=0$, $j\neq i$. Из неприводимости набора
$R_{1},\ldots,R_{m}$ следует, что $H_0=\mathbb{C}^{1}$.
Используя равенство~\eqref{E:KerB}, получим $\dim\mathrm{Ker}\,B=m$. Поэтому
$\dim H=(N+1)-m=m+r+1$, $\dim H_{k}=1$ для всех $1\leqslant k\leqslant
N+1$.  Поэтому обобщенная размерность системы $S$ равна $(m+r+1;1)$.

Таким образом, мы получили
$|M_e|$ неприводимых унитарно неэквивалентных ненулевых систем подпространств
$S$, удовлетворяющих
(Ang), (Com), (Ort), соответствующих элементам
$i\in M_{e}$.
Осталось рассмотреть случай, когда
$R_{i}=0$ для всех
$i\in M_e$. Тогда
\eqref{E:sumort} можно переписать в виде
\begin{equation}\label{E:sumort1}
\sum_{k\in M_{l}}\tau_{k}^{2}R_{k}\leqslant\xi(\tau)I.
\end{equation}
Рассмотрим следующие варианты для
$|M_{l}|$.

\textbf{1.} Пусть $|M_{l}|\geqslant 3$.  Хорошо известно, что задача
описания с точностью до унитарной эквивалентности неприводимой
$n\text{-}$\hspace{0pt}ки ортопроекторов при $n\geqslant 3$ не менее
сложна, чем задача описания с точностью до унитарной эквивалентности
неприводимой пары ограниченных самосопряженных операторов, т.е
является $\ast\text{-}$\hspace{0pt}дикой (см., например,
\cite{OstSam}).  Для исследования неравенства~\eqref{E:sumort1} нам
понадобиться <<усиленный>> вариант утверждения о
$\ast\text{-}$\hspace{0pt}дикости.

\begin{lemma}\label{lemma:wild}
  Для любого $\varepsilon>0$ задача описания неприводимых троек
  ортопроекторов $R_{1}$, $R_{2}$, $R_{3}$ с точностью до унитарной
  эквивалентности, удовлетворяющих условию
  \begin{equation*}
    R_{1}+R_{2}+R_{3}\leqslant(1+\varepsilon)I
  \end{equation*}
  является $\ast\text{-}$\hspace{0pt}дикой задачей.
\end{lemma}

\begin{proof}
  Пусть гильбертово пространство $L=l_{2}$.
  Определим $K=L\oplus L\oplus L$. Для трех операторов $B_{1,2},B_{1,3},B_{2,3}$ в $L$ определим
  тройку подпространств $K$
  \begin{equation*}
    K_{1}=\{(x,0,0)\mid x\in L\},\quad
    K_{2}=\{(B_{1,2}x,x,0) \mid x\in L\},\quad
    K_{3}=\{(B_{1,3}x,B_{2,3}x,x) \mid x\in L\}.
  \end{equation*}
   Пусть $R_{i}$ "--- ортопроектор на $K_{i}$, $i=1,2,3$.
   Если
  $\|B_{1,2}\|,\|B_{1,3}\|,\|B_{2,3}\|\to 0$, то
  $\|R_{1}+R_{2}+R_{3}-I_{K}\|\to 0$
  (это следует из формул для $R_{i}$, приведенных в \cite{Sunder}, лемма 3).
  Поэтому существует $\alpha>0$, такое, что если
  нормы $\|B_{1,2}\|,\|B_{1,3}\|,\|B_{2,3}\|$ не больше
  $2\alpha$, то $R_{1}+R_{2}+R_{3}\leqslant(1+\varepsilon)I_{K}.$

  Для пары самосопряженных операторов $A,B:L\to L$ с нормами
  $\|A\|\leqslant\alpha$ и $\|B\|\leqslant\alpha$ определим
  $B_{1,2}=\alpha I_{L},\,B_{2,3}=\alpha I_{L},\,B_{1,3}=A+iB.$
  Проверим справедливость следующих утверждений:

\textbf{1.}
    Если пара $\{A,B\}$ неприводима, то тройка ортопроекторов
    $\{R_{1},R_{2},R_{3}\}$ неприводима. Действительно, предположим
    противное. Тогда существует унитарный оператор $U:K\to K$,
    отличный от $\lambda I_{K},\,|\lambda|=1$, такой, что
    $UR_{i}=R_{i}U,\,i=1,2,3.$ Из теоремы 2 работы~\cite{Sunder} следует, что
    $U=\mathrm{diag}\,(U_{1},U_{2},U_{3})$, где унитарные операторы
    $U_{i}:L\to L$ удовлетворяют
    и $U_{j}B_{j,k}=B_{j,k}U_{k}$ при
    $j<k.$ Подставляя $(j,k)=(1,2)$ имеем $U_{1}=U_{2}.$ Подставляя
    $(j,k)=(2,3)$ имеем $U_{2}=U_{3}.$ Из равенства
    $U_{1}B_{1,3}=B_{1,3}U_{3}$ получаем: $U_{1}A=AU_{1}$ и
    $U_{1}B=BU_{1}$. Поэтому $U_{1}=\lambda I_{L},\,|\lambda|=1,$ а
    тогда $U=\lambda I_{K}.$ Получили противоречие.

\textbf{2.}
    Если пары $\{A,B\}$ и $\{A^{\prime},B^{\prime}\}$ унитарно
    неэквивалентны, то построенные по ним тройки ортопроекторов
    $\{R_{1},R_{2},R_{3}\}$ и
    $\{R_{1}^{\prime},R_{2}^{\prime},R_{3}^{\prime}\}$ унитарно
    неэквивалентны. Действительно, предположим противное. Тогда
    существует унитарный оператор $U:K\to K$, такой, что
    $UR_{i}=R_{i}'U,\,i=1,2,3.$ Из теоремы 2 работы~\cite{Sunder} следует, что
    $U=\mathrm{diag}\,(U_{1},U_{2},U_{3})$, где унитарные операторы
    $U_{i}:L\to L$ удовлетворяют
    $U_{j}B_{j,k}=B_{j,k}'U_{k}$ для всех
    $j<k$. Подставив $(j,k)=(1,2)$ имеем $U_{1}=U_{2}.$ Подставив
    $(j,k)=(2,3)$ имеем $U_{2}=U_{3}.$ Из равенства
    $U_{1}B_{1,3}=B_{1,3}'U_{3}$ имеем $U_{1}AU_{1}^{*}=A^{\prime}$ и
    $U_{1}BU_{1}^{*}=B^{\prime}$.  Получили противоречие.

  Итак, начальная задача содержит <<безнадежную>> задачу описания
  неприводимых унитарно неэквивалентных пар самосопряженных операторов
  $A,B$ в пространстве $L=l_{2}$, а потому и сама <<безнадежна>>.
\end{proof}

Поскольку $|M_{l}|\geqslant 3$, то $1,2,3 \in M_l$.  Обозначим
$\tau_{max} = \max\{\tau_1, \tau_2, \tau_3\}$ и пусть $\varepsilon =
\dfrac{\xi(\tau)}{\tau_{max}^2}-1>0$.  Пусть $R_{k}=0$ при $k\in
M_{l},\,k\geqslant 4$.  Для любых трех проекторов $R_1$, $R_2$, $R_3$,
таких, что
\begin{equation*}
  R_1+R_2+R_3\leqslant (1+\varepsilon) I
\end{equation*}
получим
\begin{equation*}
  \tau_1^2R_1+\tau_2^2R_2+\tau_3^2R_3
  \leqslant\tau_{max}^2(R_1+R_2+R_3)
  \leqslant\xi(\tau)I.
\end{equation*}
Таким образом, в случае $|M_l|\geqslant3$ задача описания неприводимых
унитарно неэквивалентных наборов ортопроекторов $R_{1},\ldots,R_{m}$,
удовлетворяющих \eqref{E:sumort1}, содержит подзадачу, которая, как мы
показали в предыдущей лемме, <<безнадежна>>, следовательно и сама
задача <<безнадежна>>.

\textbf{2.}  Пусть $M_{l}=\varnothing$. Тогда $R_{1}=\ldots=R_{m}=0$,
т.е.  $Q_{1}=\ldots=Q_{m}=I$. Поскольку набор $R_{1},\ldots,R_{m}$
неприводим, $H_{0}=\mathbb{C}^{1}$. Используя равенство
\eqref{E:KerB}, получим $\dim\mathrm{Ker}\,B=m$. Тогда $\dim
H=(N+1)-m=m+r+1$; $\dim H_{k}=1$ для всех $1\leqslant k\leqslant N+1$.
Таким образом, обобщенная размерность системы $S$ равна $(m+r+1;1)$.

\textbf{3.}  Пусть $|M_{l}|=1$. Тогда \eqref{E:sumort1} принимает вид
$\tau_{1}^{2}R_{1}\leqslant\xi(\tau)I$. Поскольку
$\tau_{1}^{2}<\xi(\tau)$, это неравенство выполнено для любого
ортопроектора $R_{1}$. Поскольку набор $R_{1},\ldots,R_{m}$ неприводим
и $R_{2}=\ldots=R_{m}=0$, то $H_{0}=\mathbb{C}^{1}$ и либо $R_{1}=0$,
либо $R_{1}=I$.

В случае $R_{1}=0$ имеем $Q_{1}=\ldots=Q_{m}=I$. Используя равенство
\eqref{E:KerB}, получим $\dim\mathrm{Ker}\,B=m$. Тогда $\dim
H=(N+1)-m=m+r+1$, $\dim H_{k}=1$ для всех $1\leqslant k\leqslant
N+1$. Поэтому обобщенная размерность системы $S$ равна $(m+r+1;1)$.

В случае $R_{1}=I$, $R_{2}=\ldots=R_{m}=0$, имеем  $Q_{1}=0$,
$Q_{2}=\ldots=Q_{m}=I$. Используя равенство \eqref{E:KerB}, получим
$\dim\mathrm{Ker}\,B=m-1$. Поэтому $\dim H=(N+1)-(m-1)=m+r+2$, $\dim
H_{k}=1$ при всех $1\leqslant k\leqslant N+1$. Обобщенная размерность
системы $S$ равна $(m+r+2;1)$.

\textbf{4.}  Пусть $|M_{l}|=2$.  В этом случае
неравенство~\eqref{E:sumort1} будет записано как
\begin{equation}\label{E:sumort:2}
  \tau_1^2R_1+\tau_2^2R_2\leqslant \xi(\tau)I.
\end{equation}

Хорошо известно следующее описание всех (не обязательно удовлетворяющих \eqref{E:sumort:2})
неприводимых пар ортопроекторов $R_{1},R_{2}$ в гильбертовом пространстве $H_{0}$
с точностью до унитарной эквивалентности (см., например, \cite{Halmos}).
Все такие пары ортопроекторов можно разделить на
\begin{enumerate}
\item четыре неприводимых пары ортопроекторов в $H_{0}=\mathbb{C}^{1}$:
  \begin{align*}
    \pi_{00}&:\, R_1 =0, R_2=0; \qquad
    \pi_{01}:\, R_1=0, R_2=I; \\
    \pi_{10}&:\, R_1 =I, R_2=0; \qquad
    \pi_{11}:\, R_1=I, R_2=I;
  \end{align*}
\item семейство пар $\pi_\varphi$, $\varphi\in(0,\pi/2)$ в $H_{0}=\mathbb{C}^{2}$:
\begin{equation}\label{E:two_projections}
  R_1 = \begin{pmatrix}
    1&0\\
    0&0
  \end{pmatrix},\qquad
  R_2 = \begin{pmatrix}
    \cos^2\varphi          & \cos\varphi\sin\varphi\\
    \cos\varphi\sin\varphi & \sin^2\varphi
  \end{pmatrix},
\end{equation}
\end{enumerate}

Из этих пар нам нужно выбрать те, для которых выполнено
неравенство~\eqref{E:sumort:2}. Рассмотрим следующие варианты.

\textbf{4.1.} В случае $H_{0}=\mathbb{C}^{1}$, $R_{1}=R_{2}=0$
неравенство~\eqref{E:sumort:2} выполнено, обобщенная размерность
системы $S$ равна $(m+r+1;1)$, так как $\dim\mathrm{Ker}\,B= m$.

\textbf{4.2.} Если $H_{0}=\mathbb{C}^{1}$, $R_{1}=I$ и $R_{2}=0$ или
$R_{1}=0$ и $R_{2}=I$, то неравенство~\eqref{E:sumort:2} выполнено и
обобщенная размерность системы $S$ равна $(m+r+2;1)$, так как
$\dim\mathrm{Ker}\,B= m-1$

\textbf{4.3.} Пусть $H_{0}=\mathbb{C}^{1}$, $R_{1}=R_{2}=I$.
Неравенство \eqref{E:sumort:2} выполнено тогда и только тогда, когда
$\xi(\tau)\geqslant\tau_{1}^{2}+\tau_{2}^{2}$. В этом случае
обобщенная размерность $S$ равна
\begin{enumerate}
\item
$(m+r+2;1)$ если $\xi(\tau)=\tau_{1}^{2}+\tau_{2}^{2}$, так как $\dim\mathrm{Ker}\,B= m-1$;
\item
$(m+r+3;1)$ если $\xi(\tau)>\tau_{1}^{2}+\tau_{2}^{2}$, так как $\dim\mathrm{Ker}\,B= m-2$.
\end{enumerate}

\textbf{4.4.}  Рассмотрим ситуацию, когда $H_{0}=\mathbb{C}^{2}$ и
ортопроекторы $R_{1},R_{2}$ заданы формулами
\eqref{E:two_projections}, $\varphi\in(0,\pi/2)$.  Неравенство
\eqref{E:sumort:2} выполнено тогда и только тогда, когда $(2\times
2)\text{-}$\hspace{0pt}матрица (оператор)
\begin{equation*}
M=\xi(\tau)I-\tau_{1}^{2}R_{1}-\tau_{2}^{2}R_{2}=\begin{pmatrix}
\xi(\tau)-\tau_{1}^{2}-\tau_{2}^{2}\cos^{2}\varphi  & -\tau_{2}^{2}\cos\varphi \sin\varphi\\
-\tau_{2}^{2}\cos\varphi \sin\varphi                & \xi(\tau)-\tau_{2}^{2}\sin^{2}\varphi
\end{pmatrix}
\end{equation*}
неотрицательно определена. Это равносильно тому, что диагональные
элементы и определитель матрицы $M$ неотрицательны. Элемент
$(M)_{1,1}\geqslant 0$ тогда и только тогда, когда
\begin{equation*}
\cos^{2}\varphi\leqslant\frac{\xi(\tau)-\tau_{1}^{2}}{\tau_{2}^{2}}.
\end{equation*}
Элемент $(M)_{2,2}>0$ при любом $\varphi$. Легко проверить, что
определитель
\begin{equation*}
\det M=(\xi(\tau)-\tau_{1}^2)(\xi(\tau)-\tau_2^2)-\tau_1^2\tau_2^2\cos^{2}\varphi.
\end{equation*}
Поэтому условие
$\det M\geqslant 0$ можно переписать в виде
\begin{equation*}
\cos^{2}\varphi\leqslant\frac{\xi(\tau)-\tau_{1}^{2}}{\tau_{2}^{2}}\frac{\xi(\tau)-\tau_{2}^{2}}{\tau_{1}^{2}}=\eta(\tau).
\end{equation*}

Рассмотрим следующие варианты.

\textbf{4.4.1.}  Предположим, что $\xi(\tau)\geqslant
\tau_{1}^{2}+\tau_{2}^{2}$. Тогда для произвольного
$\varphi\in(0,\pi/2)$ матрица $M$ положительно определена. Используя
формулу \eqref{E:KerB}, получаем
$\dim\mathrm{Ker}\,B=1+1+2(m-2)=2m-2$. Поэтому $\dim
H=2(N+1)-(2m-2)=2m+2r+4$, $\dim H_{k}=\dim H_{0}=2$ при всех
$1\leqslant k\leqslant N+1$. Обобщенная размерность системы $S$ равна
$(2m+2r+4;2)$.

\textbf{4.4.2.}  Предположим, что
$\xi(\tau)<\tau_1^2+\tau_2^2$. Определим угол
\begin{equation*}
  \varphi(\tau)=\arccos\sqrt{\eta(\tau)}\in(0,\pi/2).
\end{equation*}
Матрица $M$ неотрицательно определена тогда и только тогда, когда
$\varphi\in[\varphi(\tau),\pi/2)$.

Если $\varphi\in(\varphi(\tau),\pi/2)$, то $M$ положительно определена
и обобщенная размерность $S$ равна $(2m+2r+4;2)$.

Если $\varphi=\varphi(\tau)$, то $\mathrm{Ker}\,M$ одномерно, поэтому
$\dim\mathrm{Ker}\,B=2m-1$ и обобщенная размерность $S$ равна
$(2m+2r+3;2)$.


\subsection{Классификационная теорема}

Сформулируем результаты, полученные в подразделе~\ref{SS:irred} в виде
теоремы.

\begin{theorem}\label{T:main}
Если $\xi(\tau)<0$, то не существует ненулевой системы
подпространств~$S$, удовлетворяющей (Ang), (Com), (Ort).

В случае $\xi(\tau)=0$ с точностью до унитарной эквивалентности
существует единственная ненулевая неприводимая система
подпространств~$S$, удовлетворяющая условиям (Ang), (Com), (Ort). Ее
обобщенная размерность равна~$(m+r;1)$.

В случае $\xi(\tau)>0$ с точностью до унитарной эквивалентности все
ненулевые неприводимые системы подпространств~$S$, удовлетворяющие
(Ang), (Com), (Ort), описываются следующим образом.
\begin{enumerate}

\item $M_{l}=\varnothing:$
  \begin{enumerate}
  \item $|M_{e}|+1$ систем обобщенной размерности $(m+r+1;1)$.
  \end{enumerate}

\item $|M_{l}|=1:$
  \begin{enumerate}
  \item $|M_{e}|+1$ систем обобщенной размерности $(m+r+1;1);$
  \item одна система обобщенной размерности $(m+r+2;1)$.
  \end{enumerate}

\item $|M_{l}|=2$, $\sum\limits_{k\in M_{l}}\tau_{k}^{2}>\xi(\tau):$
  \begin{enumerate}
  \item $|M_{e}|+1$ систем обобщенной размерности $(m+r+1;1);$
  \item две системы обобщенной размерности $(m+r+2;1);$
  \item бесконечная семья систем обобщенной размерности $(2m+2r+4;2)$,
    параметризованная углом $\varphi\in(\varphi(\tau),\pi/2)$, где
    $\varphi(\tau)\in(0,\pi/2);$
  \item одна система обобщенной размерности $(2m+2r+3;2)$,
    соответствующая углу $\varphi=\varphi(\tau)$.
  \end{enumerate}

\item $|M_{l}|=2$ и $\sum\limits_{k\in M_{l}}\tau_{k}^{2}=\xi(\tau):$
  \begin{enumerate}
  \item $|M_{e}|+1$ систем обобщенной размерности $(m+r+1;1);$
  \item три системы обобщенной размерности $(m+r+2;1);$
  \item бесконечная семья систем обобщенной размерности $(2m+2r+4;2)$,
    параметризованная углом $\varphi\in(0,\pi/2)$.
  \end{enumerate}

\item $|M_{l}|=2$ и $\sum\limits_{k\in M_{l}}\tau_{k}^{2}<\xi(\tau):$
  \begin{enumerate}
  \item $|M_{e}|+1$ систем обобщенной размерности $(m+r+1;1);$
  \item две системы обобщенной размерности $(m+r+2;1);$
  \item одна система обобщенной размерности $(m+r+3;1);$
  \item бесконечная семья систем обобщенной размерности $(2m+2r+4;2)$,
    параметризованная углом $\varphi\in(0,\pi/2)$.
\end{enumerate}

\end{enumerate}

Если $|M_{l}|\geqslant 3$, то задача описания всех неприводимых
унитарно неэквивалентных систем $S$, удовлетворяющих (Ang), (Com),
(Ort), является $\ast\text{-}$\hspace{0pt}дикой.
\end{theorem}

\begin{corollary}
  Если $m\geqslant 3$ и $\tau_{k}<1/\sqrt{m+r+1}$ для
  $k=1,2,\ldots,m+r$, то задача описания всех неприводимых унитарно неэквивалентных
  систем подпространств $S$, удовлетворяющих (Ang), (Com), (Ort), является
  $\ast\text{-}$\hspace{0pt}дикой.
\end{corollary}

\subsection{Пример}
В качестве примера дадим описание с точностью до унитарной
эквивалентности всех ненулевых неприводимых систем
подпространств~$S=(H;H_{1},\ldots,H_{8})$, удовлетворяющих условиям
(Ang), (Com), (Ort), в случае, когда $m=3$, $r=1$, а функция $\tau =
\tau(\tau_0)$, $\tau_{0}\in(0,1/3)$, задана равенствами
$\tau_{1}=\tau_{0}$, $\tau_{2}=\sqrt{2}\tau_{0}$,
$\tau_{3}=2\tau_{0}$, $\tau_{4}=3\tau_{0}$. В рассматриваемом случае
$\xi(\tau)=1-16\tau_{0}^{2}$.

Если $\tau_{0}\in(0,1/\sqrt{20})$, то задача описания всех
неприводимых унитарно неэквивалентных систем $S$, удовлетворяющих
(Ang), (Com), (Ort), является $\ast\text{-}$\hspace{0pt}дикой.

На отрезке $\tau_{0}\in[1/\sqrt{20}, 1/4]$ с точностью до унитарной эквивалентности все
ненулевые неприводимые системы подпространств~$S$, удовлетворяющие
(Ang), (Com), (Ort), описываются следующим образом.
\begin{enumerate}
\item $\tau_{0}=1/\sqrt{20}:$
  \begin{enumerate}
  \item две системы обобщенной размерности $(5;1);$
  \item две системы обобщенной размерности $(6;1);$
  \item одна система обобщенной размерности $(7;1);$
  \item бесконечная семья систем обобщенной размерности $(12;2)$,
    параметризованная углом $\varphi\in(0,\pi/2)$.
  \end{enumerate}

\item $\tau_{0}\in(1/\sqrt{20},1/\sqrt{19}):$
  \begin{enumerate}
  \item одна система обобщенной размерности $(5;1);$
  \item две системы обобщенной размерности $(6;1);$
  \item одна система обобщенной размерности $(7;1);$
  \item бесконечная семья систем обобщенной размерности $(12;2)$,
    параметризованная углом $\varphi\in(0,\pi/2)$.
  \end{enumerate}

\item $\tau_{0}=1/\sqrt{19}:$
  \begin{enumerate}
  \item одна система обобщенной размерности $(5;1);$
  \item три системы обобщенной размерности $(6;1);$
  \item бесконечная семья систем обобщенной размерности
    $(12;2)$, параметризованная углом
    $\varphi\in(0,\pi/2)$.
  \end{enumerate}

\item $\tau_{0}\in(1/\sqrt{19},1/\sqrt{18}):$
  \begin{enumerate}
  \item одна система обобщенной размерности $(5;1);$
  \item две системы обобщенной размерности $(6;1);$
  \item бесконечная семья систем обобщенной размерности $(12;2)$,
    параметризованная углом $\varphi\in(\varphi(\tau),\pi/2)$, где
    \begin{equation*}
      \varphi(\tau)=
      \arccos\left(\frac{(1-17\tau_{0}^{2})(1-18\tau_{0}^{2})}{2\tau_{0}^{4}}\right)^{1/2}
      \in(0,\frac{\pi}{2}),
    \end{equation*}
  \item одна система обобщенной размерности $(11;2)$. Этой системе
    соответствует угол $\varphi=\varphi(\tau)$.
  \end{enumerate}

\item $\tau_{0}=1/\sqrt{18}:$
  \begin{enumerate}
  \item две системы обобщенной размерности $(5;1);$
  \item одна система обобщенной размерности $(6;1).$
  \end{enumerate}

\item $\tau_{0}\in(1/\sqrt{18},1/\sqrt{17}):$
  \begin{enumerate}
  \item одна система обобщенной размерности $(5;1);$
  \item одна система обобщенной размерности $(6;1)$.
  \end{enumerate}

\item $\tau_{0}=1/\sqrt{17}:$
  \begin{enumerate}
  \item две системы обобщенной размерности $(5;1);$
  \end{enumerate}

\item $\tau_{0}\in(1/\sqrt{17},1/4):$
  \begin{enumerate}
  \item одна система обобщенной размерности $(5;1);$
  \end{enumerate}

\item $\tau_{0}=1/4:$
  \begin{enumerate}
  \item одна система обобщенной размерности $(4;1);$
\end{enumerate}

\end{enumerate}

Если $\tau_{0}\in(1/4,1/3)$, то не существует ненулевой системы
подпространств $S$, удовлетворяющей (Ang), (Com), (Ort).

\end{document}